\documentclass[a4paper,10pt]{article}
\usepackage[utf8]{inputenc}

\usepackage{comment,xcomment,amssymb,amsmath,color}
\usepackage{rotating,xypic,array}
\usepackage{latexsym,mathtools,amsthm}
\usepackage{booktabs}  
\usepackage{url} 
\usepackage[section]{placeins}
\usepackage{multirow,rotating,graphicx}
\usepackage[a4paper]{geometry}
\usepackage[colorlinks]{hyperref}
\usepackage{enumitem}
\newlist{compactenum}{enumerate}{4}
\setlist[compactenum,1]{nolistsep}
\usepackage{xspace}

\title{Uniqueness of ad-invariant metrics}
\author{Diego Conti, Viviana del Barco and Federico A. Rossi}

\makeatletter
\def\namedlabel#1#2{\begingroup
   \def\@currentlabel{#2}%
   \label{#1}\endgroup
}
\makeatother

\newtheorem{theorem}{Theorem}[section]
\newtheorem{lemma}[theorem]{Lemma}

\newtheorem{corollary}[theorem]{Corollary}
\newtheorem{proposition}[theorem]{Proposition}
\theoremstyle{definition}
\newtheorem{definition}[theorem]{Definition}
\newtheorem{example}[theorem]{Example}

\theoremstyle{remark}
\newtheorem{remark}[theorem]{Remark}

\newcommand{\R}{\mathbb{R}}
\newcommand{\im}{\mathrm{Im}\,}         
\newcommand{\lie}[1]{\mathfrak{#1}}     
\newcommand{\g}{\lie{g}}
\newcommand{\Z}{\mathbb{Z}}

\newcommand{\N}{\mathbb{N}}

\newcommand{\C}{\mathbb{C}}
\newcommand{\K}{\mathbb{K}}            

\newcommand{\hook}{\lrcorner\,}

\newcommand{\CO}{\mathrm{CO}}

\newcommand{\so}{\mathfrak{so}}
\newcommand{\co}{\mathfrak{co}}

\newcommand{\su}{\mathfrak{su}}

\newcommand{\id}{\mathrm{Id}}   

\DeclareMathOperator{\Ann}{Ann}
\newcommand{\gl}{\lie{gl}}

\newcommand{\Span}[1]{\operatorname{Span}\left\{#1\right\}}

\newcommand{\tran}[1]{\hspace{.2mm}\prescript{t\hspace{-.5mm}}{}{#1}}
\DeclareMathOperator{\ric}{ric}  

\DeclareMathOperator{\End}{End}
\DeclareMathOperator{\Aut}{Aut}

\DeclareMathOperator{\diag}{diag}
\DeclareMathOperator{\Der}{Der}

\DeclareMathOperator{\ad}{ad}

\DeclareMathOperator{\Tr}{tr}

\newcolumntype{C}{>{$}c<{$}}
\newcolumntype{L}{>{$}l<{$}}
\newcolumntype{R}{>{$}r<{$}}

\newcommand{\st}{\;\mid\;}          

\newcommand{\mm}{\mathfrak{m}}

\newcommand{\mg}{\mathfrak{g}}
\newcommand{\mb}{\mathfrak{b}}
\newcommand{\mq}{\mathfrak{q}}

\newcommand{\md}{\mathfrak{d}}
\newcommand{\mn}{\mathfrak{n}}
\newcommand{\mz}{\mathfrak{z}}
\newcommand{\mk}{\mathfrak{k}}

\newcommand{\ms}{\mathfrak{s}}

\newcommand{\mh}{\mathfrak{h}}

\newcommand{\tmg}{\tilde{\mathfrak{g}}}

\newcommand{\rigid}{solitary\xspace}
\newcommand{\adsolitary}{$T^*$-solitary\xspace}
\newcommand{\weaklyrigid}{weakly solitary\xspace}

\newcounter{rowno}
\begin{document}
\maketitle
\begin{abstract}
We consider Lie algebras admitting an ad-invariant metric, and we study the problem of uniqueness of the ad-invariant metric up to automorphisms. This is a common feature in low dimensions, as one can observe in the known classification of nilpotent Lie algebras of dimension $\leq 7$ admitting an ad-invariant metric. We prove that uniqueness of the metric on a complex Lie algebra $\g$ is equivalent to uniqueness of ad-invariant metrics on the cotangent Lie algebra $T^*\g$; a slightly more complicated equivalence holds over the reals. This motivates us to study the broader class of Lie algebras such that the ad-invariant metric on $T^*\g$  is unique.

We prove that uniqueness of the metric forces the Lie algebra to be solvable, but the converse does not hold, as we show by constructing solvable Lie algebras with a one-parameter family of inequivalent ad-invariant metrics. We prove sufficient conditions for uniqueness expressed in terms of both the Nikolayevsky derivation and a metric counterpart  introduced in this paper.

Moreover, we prove that uniqueness always holds for irreducible Lie algebras which are either solvable  of dimension $\leq 6$ or real nilpotent of dimension $\leq 10$.
\end{abstract}

\renewcommand{\thefootnote}{\fnsymbol{footnote}}
\footnotetext{\emph{MSC class 2020}: \emph{Primary} 53C30; \emph{Secondary} 17B30, 17B40, 53C50}
\renewcommand{\thefootnote}{\arabic{footnote}}

\section*{Introduction}
Bi-invariant metrics on Lie groups have gained relevance in differential geometry, in particular because of their relation to the study of naturally reductive spaces \cite{Kos56,Ov11} and the classification of symmetric spaces \cite{KO08}. Recently, it was proved in \cite{Bagl} that all compact pseudo-Riemannian solvmanifolds arise from quotients of solvable Lie groups carrying bi-invariant metrics; an explicit construction of such quotients is carried out in \cite{Ka20,MR85}.

The algebraic counterpart of bi-invariant metrics on Lie groups are \emph{ad-invariant metrics} on Lie algebras. Recall that an ad-invariant metric $g$ on a Lie algebra $\mg$ is a symmetric and nondegenerate bilinear form satisfying
\begin{equation}\label{eq:adinvdef}
g([X,Y],Z)=-g(Y,[X,Z]),\quad \text{ for all } X,Y,Z\in\g.
\end{equation}
Every Lie group carrying a Riemannian bi-invariant metric is a product of a compact Lie group with an abelian factor. However, indefinite metrics appear more generally, as in the case of bi-invariant metrics on solvable Lie groups, or those induced by the Killing form on any semisimple Lie algebra.

Remarkably, the study of ad-invariant metrics on Lie algebras has been developed independently of bi-invariant metrics, motivated by applications in mathematics and physics.
In the more general setting of associative algebras over a field $\K$, the notion of ad-invariant metrics is instrumental in the theory of modular representation of finite groups \cite[Chapter 2]{Karpi90}. In physics, they appear in the well known Adler-Kostant-Symes construction in Hamiltonian systems and also in conformal field theory \cite{FiSt96}. Further applications can be found in the survey articles \cite{Bo97,Ovandosurvey}. Along this paper, we will be interested in the cases $\K=\R$, $\K=\C$.

Determining whether an ad-invariant metric exists on a fixed Lie algebra is a straightforward linear computation. However, the general classification of Lie algebras admitting an ad-invariant metric is a wide open problem, as is clear from \cite{Ovandosurvey}, though classifications exist for several restricted families of Lie algebras \cite{BaumKath,dBO12, FavreSantharoubane:adInvariant,Kath:NilpotentSmallDim,KathOlbrich:Metric}.
A central ingredient of \cite{BaumKath, FavreSantharoubane:adInvariant}  is the well-known double extension procedure introduced by Medina and Revoy \cite{MedinaRevoy}, which constructs a Lie algebra structure and an ad-invariant metric on the vector space obtained from the direct sum of a Lie algebra $(\md,g)$ carrying an ad-invariant metric $g$ and the cotangent of a Lie algebra acting on $\md$. Notably, they show that every irreducible Lie algebra with an ad-invariant metric is obtained in this way, except simple Lie algebras and $\R$. However, the double extension generally depends on a set of parameters, and different choices of $(\md,g)$ can yield isometric double extensions, making the classification quite hard.

A transversal line of research is the identification of Lie algebras admitting a \emph{unique} ad-invariant metric. This problem appears more scarcely in the literature, but in two different flavours. In \cite{BaBe07}, the uniqueness of ad-invariant metrics on a fixed Lie algebra is measured by the dimension of the linear space spanned by all invariant nondegenerate symmetric bilinear forms. By contrast, uniqueness up to Lie algebra automorphisms and scaling is considered in \cite{MeRe93}, where it is shown to hold under the hypothesis that every ad-invariant map is a linear combination of the identity and a derivation.

In this paper we pursue this latter point of view, and study Lie algebras such that the group of automorphisms acts transitively on the space of ad-invariant metrics. We introduce an infinitesimal characterization of this condition; precisely, we define a metric to be \emph{solitary} if every self-adjoint ad-invariant map can be written as the self-adjoint part of a derivation of the Lie algebra. We prove that if an ad-invariant metric on a Lie algebra $\g$ is solitary, then so are all ad-invariant metrics on $\g$ (Proposition~\ref{prop:rigid_independent}); we then say that the Lie algebra is solitary. In addition we show that if the Lie algebra is irreducible, being solitary is equivalent to admitting a unique ad-invariant metric up to automorphisms (and sign, if $\K=\R$). In particular, the signature of the ad-invariant metric is uniquely determined up to sign (Theorem~\ref{thm:rigidimpliesuniqueness}).

We also consider the situation in which the group of automorphisms acts on the space of ad-invariant metrics with cohomogeneity one; we show that these Lie algebras admit exactly one or two ad-invariant metrics up to automorphisms and a scalar factor, according to whether all derivations are traceless or not. In particular, we recover the uniqueness result of \cite{MeRe93}. We note that  simple Lie algebras, whilst not solitary, satisfy this weaker form of uniqueness.

The solitary property behaves well with respect to constructions. We show that a reducible Lie algebra is solitary if and only if its factors are solitary, the complexification $\g^\C$ of a Lie algebra $\g$ is solitary if and only if so is $\g$, and a complex Lie algebra $\g$ is solitary if and only if  the underlying real Lie algebra $\g^\R$ is solitary. A similar property holds for cotangents. Recall that the cotangent of \emph{any} Lie algebra admits an ad-invariant metric; a key observation for our work is that a Lie algebra admitting an ad-invariant metric is solitary if and only if its cotangent is solitary too (Theorem~\ref{thm:rigiffTrig}).

This motivates the study of Lie algebras whose cotangent is solitary, which we call $T^*$-solitary. We show that every $2$-step nilpotent Lie algebra and every Riemannian nilsoliton is $T^*$-solitary, although not all of them admit ad-invariant metrics. This concept, in combination with the double extension characterization, allows us to prove one of the main results of the paper: solitary and $T^*$-solitary Lie algebras are necessarily solvable (see Theorem~\ref{thm:solitaryimpliessolvable}). As a consequence, uniqueness (up to automorphisms and sign) of ad-invariant metrics only occurs within the class of solvable Lie algebras. However, we show with examples that there exist solvable Lie algebras admitting a one parameter family of ad-invariant metrics up to automorphisms and sign.

In order to give sufficient conditions for a Lie algebras admitting ad-invariant metrics to be solitary, we use gradings of the Lie algebra behaving well with respect to the metric. Such a grading can be defined, in particular, by a canonical derivation associated to any Lie algebra endowed with a nondegenerate scalar product (not necessarily ad-invariant). The existence of this derivation is proved in Theorem~\ref{thm:metricnik} by adapting to the metric context the techniques in the celebrated work of Nikolayevsky \cite{Nikolayevsky}; accordingly, we call it the \emph{metric Nikolayevsky derivation}. In some cases this derivation turns out to coincide with the (usual) Nikolayevsky derivation. We prove in Theorem~\ref{thm:uniquenesstheorem} that an ad-invariant metric whose metric Nikolayevsky derivation has positive eigenvalues is solitary. Whilst this only occurs in the presence of a positive grading, which forces the Lie algebra to be nilpotent, a more general result holds in the case of nonnegative eigenvalues, under the assumption that the zero eigenspace is $T^*$-solitary. Beyond the ad-invariant context, we expect this general construction to find applications in other areas of homogeneous pseudo-Riemannian geometry.

As an application of the above results, we obtain that many low dimensional solvable Lie algebras are $T^*$-solitary. Precisely, we prove that every solvable Lie algebra up to dimension $6$ and every nilpotent Lie algebra up to dimension $10$ is $T^*$-solitary.  In both cases we use an inductive argument on non $T^*$-solitary Lie algebras in combination with: on the one hand, the fact that nilpotent Lie algebras of dimension $\leq 10$ admitting ad-invariant metrics (classified in \cite{Kath:NilpotentSmallDim}) are \emph{nice} in the sense of \cite{LauretWill:EinsteinSolvmanifolds,Nikolayevsky} and thus their metric Nikolayevsky derivations can be easily computed; on the other hand, the description of solvable Lie algebras of dimension $\leq 6$ as double extensions of abelian Lie algebras.

\smallskip

A brief account of the organization of the paper is as follows: Section~\ref{sec:Decomposable} sets notation, and reviews properties of the decomposition of Lie algebras admitting ad-invariant metrics. Section~\ref{sec:uniqueness} introduces the problem  of uniqueness of ad-invariant metrics and relates it to self-adjoint ad-invariant maps, whose properties are also studied. In addition, it introduces the definition of (weakly) solitary metrics, and settles their link with the orbits of the action of the automorphisms group on the set of ad-invariant metrics.
The relation between the solitary condition for a real Lie algebra and its complexification is established in Section~\ref{sec:complexification}. Section~\ref{sec:Canonical} introduces the gradings of our interest for Lie algebras with nondegenerate scalar products, and contains the construction of the metric Nikolayevsky derivation. Section~\ref{sec:Cotangents} deals with cotangent Lie algebras endowed with their canonical ad-invariant metric and considers the solitary condition on them, introducing the concept of $T^*$-solitary Lie algebras. In Section~\ref{sec:DE} we use the double extension procedure, in combination with the results in Section~\ref{sec:Cotangents}, to show that $T^*$-solitary Lie algebras are always solvable. Moreover, we prove that in particular cases the eigenspaces of the metric Nikolayevsky derivation provide a description of the Lie algebra as a special double extension. Section~\ref{sec:lowdim} contains the study of the $T^*$-solitary condition in low dimension.

\medskip

\noindent \textbf{Acknowledgments:} The authors express their gratitude to Said Benayadi for pointing out relevant references on the subject. D.~Conti and F.A.~Rossi acknowledge GNSAGA of INdAM. F.A.~Rossi also acknowledges the Young Talents Award of Universit\`{a} degli Studi di Milano-Bicocca joint with Accademia Nazionale dei Lincei.

\section{Lie algebras with ad-invariant metrics}\label{sec:Decomposable}
This section aims to fix notation and establish some elementary facts about the decomposition of Lie algebras endowed with ad-invariant metrics; notice that similar results have been proved in \cite{As79,Bo97,ZhZh01}, but not exactly in the form that will be  needed in the rest of this paper; for this reason, we will give an almost self-contained account rather than refer to existing results.

We define a \emph{metric Lie algebra} as a pair $(\mg,g)$, where $\mg$ is a finite-dimensional Lie algebra and $g$ is a metric on $\mg$, i.e. a nondegenerate scalar product. The metric $g$ is said to be \emph{ad-invariant} if~\eqref{eq:adinvdef} holds; notice that for some authors this condition is part of the definition of metric Lie algebra.

Fix a metric Lie algebra $(\mg,g)$, with $g$ ad-invariant metric. Given a subspace $\mm\subset \mg$, we denote by $\mm^\bot$ the orthogonal of $\mm$ with respect to $g$, namely,
\[
\mm^\bot=\{X\in \mg \st g(X,Y)=0, \text{ for all } Y\in \mg\}.
\]
A subspace $\mm\subset \mg$ is called nondegenerate if $\mm\cap \mm^\bot=\{0\}$. Using the ad-invariance of the metric~\eqref{eq:adinvdef}, one can easily check that if $\mh$ is an ideal of $\mg$, then $\mh^\bot$ is also an ideal, and $\mh$ is nondegenerate if and only if so is $\mh^\bot$. Moreover, if $\mz(\mg)$ denotes the center of $\mg$ and $\mg'$ denotes the commutator of $\mg$, then it is straightforward to check that $(\mg')^\bot=\mz(\mg)$.

Recall that a Lie algebra $\mg$ is called \emph{reducible} if it can be written as a direct sum of two nontrivial ideals, $\mg=\mh\oplus \mk$, and \emph{irreducible} otherwise. If $\mg$ is endowed with an ad-invariant metric $g$, then we say that $(\mg,g)$ is \emph{reducible as a metric Lie algebra} if it can be written as the sum of two nondegenerate orthogonal ideals, namely, $\mg=\mh\oplus \mh^\bot$, and \emph{irreducible as a metric Lie algebra} otherwise.

Notice that $(\mg,g)$ is irreducible as a metric Lie algebra if and only if every ideal strictly contained in $\mg$ is degenerate.

It is clear that if $\mg$ is irreducible, then $(\mg,g)$ is irreducible as a metric Lie algebra for every ad-invariant metric $g$. We shall study the converse.

\begin{remark}\label{rem:centreduc}
Notice that every nonabelian Lie algebra $\mg$ whose center $\mz(\mg)$ is not contained in the commutator $\mg'$ is reducible. Indeed, if $\mm$ is a complement of $\mg'\cap \mz(\mg)$ inside $\mz(\mg)$, then $\mm\cap \mg'=\{0\}$ and thus there is a complementary subspace $\tmg$ of $\mg$ such that $\mg=\mm\oplus\tmg$ and $\mg'\subset \tmg$. As a consequence, $\mm$ and $\tmg$ are nontrivial ideals of $\mg$.
\end{remark}

We shall prove that a similar reduction can be obtained for Lie algebras with ad-invariant metrics whose center is not contained in the commutator (see also \cite[Remark 1]{dBOV}).

\begin{lemma}\label{lem:mplustildeg}
Let $\mg$ be a nonabelian Lie algebra with an ad-invariant metric $g$. If the center of $\mg$ is not contained in its commutator, then $(\mg,g)$ is reducible as a metric Lie algebra as $\mg=\mm\oplus\tmg$, where $\mm$ is an abelian ideal and the center of $\tilde\mg$ is contained in its commutator.
\end{lemma}
\begin{proof}

Let $\mm$ be an arbitrary complement of $\mg'\cap \mz(\mg)$ inside $\mz(\mg)$, so that
\begin{equation}\label{eq:mgz}
\mz(\mg)=\mm\oplus (\mg'\cap \mz(\mg));
\end{equation}
since $\mz(\mg)$ is not contained in $\mg'$, $\mm$ is not trivial. As $\mm$ is contained in $\mz(\mg)$, which is orthogonal to $\mg'$, we have that $\mm\bot (\mg'\cap \mz(\mg))$. It follows that $\mm$ is nondegenerate; indeed, if $X\in \mm$ is orthogonal to $\mm$, then it is orthogonal to all of $\mz(\g)$, and therefore an element of $\mg'$.

Being a subset of the center, $\mm$ is an ideal, and thus $\tilde\mg=\mm^\bot$ is also a nondegenerate ideal. It is clear that $\mz(\tilde \mg)=\mg'\cap \mz(\mg)\subset \mg'= \tmg'$.
\end{proof}

\begin{corollary}\label{cor:reductive}
Let $\mg$ be a reductive Lie algebra. Then for every ad-invariant metric $g$, $\mg$ can be written as an orthogonal direct sum $\mg=\mg'\oplus\mz(\mg)$.
\end{corollary}

Let $\g$ be a Lie algebra with  an ad-invariant metric $g$ and let $\mh$ be an ideal of  $\mg$. As mentioned before, if the restriction of the metric $g$ to $\mh$ is nondegenerate, then there exists a complementary nondegenerate ideal $\mk$ such that $\mg$ decomposes as orthogonal direct sum of ideals $\mg=\mh\oplus\mk$. In fact, $\mk=\mh^\bot$ is the complementary  ideal.

The next result from \cite{ZhZh01} shows that if $\mz(\mg)\subset \mg'$, then the converse holds.

\begin{lemma}[{\cite[Lemma 3.2]{ZhZh01}}]\label{pro:iplusjgivesnondeg}
Let $\mg$ be a nonabelian Lie algebra with an ad-invariant metric $g$ such that $\mz(\mg)\subset \mg'$, and let $\mh$ be an ideal of $\mg$. If there exists a complementary ideal $\mk$ such that $\mg=\mh\oplus\mk$, then the restriction of $g$ to $\mh$ and to $\mk$ is nondegenerate. In particular, $\mh$ and $\mk$ admit ad-invariant metrics.
\end{lemma}

\begin{proposition}\label{pro:decomposable}
Let $\mg=\mg_1\oplus\mg_2$ be a reducible Lie algebra. Then $\mg$ admits an ad-invariant metric if and only if each $\mg_i$ admits one.
\end{proposition}
\begin{proof}
It is clear that if each $\mg_i$ admits an ad-invariant metric, then the orthogonal sum of such metrics defines an ad-invariant metric on $\mg$.

To prove the converse, assume that $\mg$ admits an ad-invariant metric. If $\mz(\mg)\subset \mg'$ the result follows from Lemma~\ref{pro:iplusjgivesnondeg}, since the restriction of $g$ to each ideal is nondegenerate and ad-invariant.

If $\mz(\mg)\nsubseteq \mg'$, then $\mz(\mg_i)\nsubseteq \mg_i'$ for at least one of $i=1,2$. For each $i=1,2$, consider ideals $\mm_i$, $\tmg_i$ of $\mg_i$ such that $\mg_i=\mm_i\oplus \tmg_i$
and  $\mm_i\oplus (\mz(\mg_i)\cap  \mg_i')=\mz(\mg_i)$ (see Remark~\ref{rem:centreduc}).

The ideal $\mm=\mm_1\oplus\mm_2$ is nondegenerate and verifies~\eqref{eq:mgz}, so  proceeding as in Lemma~\ref{lem:mplustildeg} we get that
its orthogonal $\tmg=\mm^\bot$ is also nondegenerate and satisfies $\mz(\tmg)\subset \tmg'$. Moreover, since
\begin{equation}\label{eq:m1mg1}
\mg=\mm_1\oplus\tmg_1\oplus\mm_2\oplus\tmg_2=\mm\oplus \tmg,
\end{equation}
we obtain that $\tmg$ is reducible and isomorphic to $\tmg_1\oplus\tmg_2$. Hence, Lemma~\ref{pro:iplusjgivesnondeg} implies that $\tmg_i$ admits an ad-invariant metric, for $i=1,2$. Since any metric on $\mm_i$ is ad-invariant (because it is abelian) we finally get that $\mg_i=\mm_i\oplus \tmg_i$ admits an ad-invariant metric, for each $i=1,2$.
\end{proof}

\begin{proposition}
\label{prop:irr_is_irr}
A Lie algebra $\mg$ admitting ad-invariant metrics is irreducible if and only if $(\mg,g)$ is irreducible as a metric Lie algebra for every ad-invariant metric $g$.
\end{proposition}
\begin{proof}
It is straightforward that if a Lie algebra is irreducible, then it is irreducible as a metric Lie algebra, for every ad-invariant metric on $\mg$.

Suppose that $(\mg,g)$ is irreducible as a metric Lie algebra with $g$ ad-invariant. By Lemma~\ref{lem:mplustildeg}, either $\mg$ is one-dimensional, or $\mz(\mg)\subset \mg'$. In the former case, $\mg$ is clearly irreducible.

In the case that $\mz(\mg)\subset \mg'$, assume for a contradiction that $\mg$ is reducible and $\mg=\mg_1\oplus\mg_2$ with $\mg_i$ nontrivial ideals. Lemma~\ref{pro:iplusjgivesnondeg} implies that $g$ restricted to each $\mg_i$ is nondegenerate. In particular $\mg=\mg_1\oplus\mg_1^\bot$, so $\mg$ is reducible as a metric Lie algebra, which is absurd.
\end{proof}

Let $\mg$ be a Lie algebra with an ad-invariant metric. From the definition of reducible metric Lie algebra, it is clear that $\mg$ can be written as an orthogonal direct sum of ideals, each of which is nondegenerate and  irreducible as a metric Lie algebra. Proposition~\ref{prop:irr_is_irr} allows us to show that these ideals are also irreducible as Lie algebras. As an immediate consequence, we obtain:

\begin{corollary}
Let $\mg$ be a Lie algebra with an ad-invariant metric $g$. Then, $\mg$ can be written as a direct sum of orthogonal ideals
\begin{equation*}
\mg=\mg_0\oplus\mg_1\oplus\cdots\oplus\mg_s,
\end{equation*}
where $\mg_0$ is a central ideal and  $\mg_1,\dots,\mg_s$ are irreducible ideals.
\end{corollary}

\begin{corollary}
Let $\mg$ be a semisimple Lie algebra, and let $\mg=\ms_1\oplus \dots \oplus \ms_k$ be its decomposition into simple ideals. Then every ad-invariant metric on $\mg$ has the form
\[h_1B_1+\dots + h_k B_k,\]
where $h_1,\dotsc, h_k$ are nonzero constants and $B_1,\dotsc, B_k$ are the Killing forms on each factor.
\end{corollary}
\begin{proof}
Fix an ad-invariant metric on $\mg$, and write $\mg$ as the orthogonal direct sum of irreducible ideals. Since every irreducible ideal in $\mg$ is one of the $\ms_i$, and ad-invariant metrics on each $\ms_i$ are multiples of the Killing form, the statement follows.
\end{proof}

\section{The uniqueness problem for ad-invariant metrics}\label{sec:uniqueness}

Given an ad-invariant metric $g$ on a Lie algebra $\mg$, it is clear that $g(\psi\cdot,\psi\cdot)$ is also ad-invariant for any  $\psi$ in $\Aut(\g)$. In other words, $\Aut(\g)$ acts on the space of ad-invariant metrics.

It is striking that for many Lie algebras admitting an ad-invariant metric (for instance, irreducible nilpotent Lie algebras of dimension $\leq 7$, see \cite{FavreSantharoubane:adInvariant}) the space of ad-invariant metrics consists of a single orbit, that is to say, the metric is unique up to automorphisms. In other cases, the ad-invariant metric is unique up to automorphisms and a scalar factor. We aim at finding a characterization of Lie algebras with these properties.

Given a Lie algebra $\g$ with a fixed metric $g$, every metric on $\g$ has the form
\[g_\phi(X,Y)= g(\phi X,Y),\]
where $\phi$ is an endomorphism of $\g$ (in the vector space sense, i.e. a linear map $\g\to\g$), self-adjoint with respect to $g$; this gives a bijection between metrics and $g$-self-adjoint endomorphisms. If $g$ is ad-invariant, then $g_\phi$ is ad-invariant if and only if
\begin{equation}
 \label{eqn:adinvariantphi}
 \phi([X,Y])=[X,\phi(Y)], \quad \text{for every }X,Y\in\g.
\end{equation}
An endomorphism satisfying~\eqref{eqn:adinvariantphi} is said to be \emph{ad-invariant}. Notice that if $\phi$ is ad-invariant and self-adjoint, then $\phi^{-1}$ has the same properties.

The group of automorphisms acts on the set of ad-invariant self-adjoint maps by $\psi\cdot\phi = \psi\phi\psi^*$, for $\psi\in \Aut(\mg)$. It follows that the uniqueness of ad-invariant metrics up to automorphisms is equivalent to the uniqueness of ad-invariant self-adjoint endomorphisms up to this $\Aut(\mg)$-action.

Recall that on a complex vector space $V$ with a real structure $\sigma$ one defines the real subspace $[V]=\{v\in V\st \sigma(v)=v\}$.
\begin{proposition}
\label{prop:eigenvaluesofphi}
Let $\g$ be a complex Lie algebra, let $\phi\colon\lie{g}\to \lie{g}$ be an ad-invariant linear map, and let $\g=\bigoplus V_\lambda$ be the decomposition into generalized eigenspaces of $\phi$. Then each $V_\lambda$ is an ideal in $\g$.

In particular given a real Lie algebra $\g$ and an ad-invariant linear map $\phi\colon\g\to\g$, in terms of the complexification $\g^\C$, we obtain a decomposition
\[\g=\bigoplus_{\begin{smallmatrix}
\lambda \text{ real}\\
\text{eigenvalue}
\end{smallmatrix}}
[V_\lambda] \oplus
\bigoplus_{\begin{smallmatrix}
\{\lambda,\overline{\lambda}\} \text{ conjugate}\\
\text{nonreal eigenvalues}\end{smallmatrix}}
[V_\lambda\oplus V_{\overline{ \lambda}}].\]
\end{proposition}
\begin{proof}
Let $\lambda$ be an eigenvalue, and $V_\lambda = \ker (\phi-\lambda \id)^n$, with $n=\dim \g$, the associated generalized eigenspace. Then, if $x\in V_\lambda$ and $y\in V_\mu$ we have
\[
(\phi-\lambda \id)^n[x,y]=[(\phi-\lambda \id)^n x,y]=0;\]
so $[x,y]$ also belongs to $V_\lambda$.

Given a real Lie algebra $\mg$ and an ad-invariant linear map $\phi\colon\mg\to\mg$, since $\phi^\C\colon\g^\C\to\g^\C$ preserves the real structure, the decomposition follows.
\end{proof}

\begin{corollary}
\label{cor:phihasoneortwoeigenvalues}
Let $g$ be an ad-invariant metric on an irreducible Lie algebra $\g$ over $\K=\R,\C$, and let  $\phi\colon\g\to\g$ be ad-invariant and self-adjoint. Then $\phi=\phi_S+\phi_N$, where
\begin{enumerate}[label=(\roman*)]
 \item $\phi_S$, $\phi_N$ are ad-invariant and self-adjoint;
 \item $\phi_N$ is nilpotent;
 \item for $\K=\C$, $\phi_S$ is a multiple of the identity;
 \item \label{it:rphi} for $\K=\R$, $\phi_S$ has the form $a\id + bJ$, with $J\colon\g\to\g$ satisfying $J^2=-\id$, and $a,b\in\R$.
\end{enumerate}
\end{corollary}
\begin{proof}
By the Jordan decomposition theorem, one can write $\phi=\phi_N+\phi_S$, with $\phi_N$ nilpotent and $\phi_S$ semisimple, where $\phi_N$ and $\phi_S$ are polynomials in $\phi$. Therefore, $\phi_N$ and $\phi_S$ are self-adjoint and ad-invariant.

In the real case, since $\g$ is irreducible, Proposition~\ref{prop:eigenvaluesofphi} applied to $\phi$ (or equivalently to $\phi_S$) implies that either $\phi$ has only one eigenvalue, which has to be real, or two complex conjugate eigenvalues $a\pm ib$. In the latter case, $\phi_S-a\id$ is diagonalizable with eigenvalues $\pm ib$, so it squares to $-b^2\id$.

The complex case is similar.
\end{proof}

\begin{remark}\label{rem:complexfromJadinv}
In the case~\ref{it:rphi} of Corollary~\ref{cor:phihasoneortwoeigenvalues}, if the component of $\phi_S=a\id + bJ$ along $J$ is not zero, then $J$ is ad-invariant and thus $(\g,J)$ is a complex vector space such that the Lie bracket is $\C$-bilinear. Therefore, this only happens when $\g$ is the real Lie algebra underlying a complex Lie algebra. Moreover, $g$ and $\phi$ induce a complex ad-invariant scalar product
\[h( X,Y) = g(X,Y) - ig(JX,Y).\]
This is $\C$-bilinear, since
\[h( JX,Y) = g(JX,Y) + ig(X,Y).\]
The two real ad-invariant scalar products  $g(\cdot,\cdot)$ and $g(J\cdot,\cdot)$ therefore correspond to the real and imaginary part of a complex ad-invariant scalar product on $(\g,J)$.
\end{remark}

Given a metric Lie algebra $(\g,g)$ and a linear map $f\colon \g\to\g$, we will denote the adjoint of $f$ by $f^*$, or $f^{*_g}$ to emphasize the dependence on the metric.
\begin{lemma}
\label{lemma:symmetrized_derivation_ad_invariant}
Let $\mg$ be a Lie algebra endowed with an ad-invariant metric $g$. Then for any derivation $D\in \Der(\mg)$, the self-adjoint endomorphism $\phi:=D+D^*$ is ad-invariant.
\end{lemma}
\begin{proof}For any $X,Y,Z$ in $\mg$, ad-invariance of the metric and the fact that $D$  is a derivation give:
\begin{align*}
g(\phi[X,Y],Z)&=g([DX,Y],Z)+g([X,DY],Z)+g([X,Y],DZ)\\
&=g(DX,[Y,Z])+g(X,[DY,Z])+g(X,[Y,DZ])=g(\phi X,[Y,Z])=g([\phi X,Y],Z).
\end{align*}
Nondegeneracy implies $\phi[X,Y]=[\phi X,Y]$, and thus $\phi$ is ad-invariant.
\end{proof}

We will see that, given a Lie algebra with an ad-invariant metric, the uniqueness of the metric is related to whether every ad-invariant self-adjoint map is the self-adjoint part of a derivation. For this reason we introduce the following:
\begin{definition}
Given a Lie algebra $\g$ over $\K=\R,\C$ and an ad-invariant metric $g$, we say that $g$ is \emph{\rigid} if every ad-invariant self-adjoint endomorphism $\phi\colon\g\to\g$ has the form $D+D^{*_g}$, for some derivation $D$ of $\mg$. We say that $g$ is \emph{\weaklyrigid} if every ad-invariant self-adjoint endomorphism $\phi\colon\g\to\g$ has the form $D+D^*+\lambda\id$ for some $\lambda\in \K$ and some derivation $D$.
\end{definition}

\begin{remark}\label{rem:phiWZ}
It is straightforward to check that endomorphisms $\phi:\mg\to\mg$ that satisfy
\begin{equation}
\label{eqn:phiWZ}
\im\phi\subset Z,\quad \ker\phi\supset\g'
\end{equation}
are also derivations. Therefore, if all $g$-self-adjoint ad-invariant endomorphisms on $\mg$ satisfy~\eqref{eqn:phiWZ}, the metric $g$ is trivially \rigid.

In addition, if every ad-invariant map can be written as a linear combination of the identity and an endomorphism satisfying~\eqref{eqn:phiWZ}, then $g$ is trivially weakly solitary. This should be compared with \cite[Theorem 3.4]{MeRe93} (see also Remark~\ref{rem:ComparsionWithThm34MeRe}).
\end{remark}

\begin{remark}
\label{remark:tracess_not_rigid}
On a Lie algebra $\g$, the identity map $\id$ is ad-invariant and self-adjoint with respect to any ad-invariant metric. Therefore, given an ad-invariant \rigid metric, there is a derivation $D$ such that $D+D^*=\id$. This implies that $D$ has nonzero trace because $\Tr(D)+\Tr(D^*)=2\Tr(D)$. Thus, Lie algebras whose derivations are traceless cannot be \rigid.
\end{remark}

\begin{remark}
\label{rk:weakimpliesirred}
If all the derivations of a  Lie algebra $\g$ are traceless and $\g$ is reducible, then it does not admit any \weaklyrigid ad-invariant metrics.  Indeed, suppose that $\g$ has an ad-invariant metric $g$; by Proposition~\ref{prop:irr_is_irr}, it is also reducible as a metric Lie algebra, say $\g=\g_1\oplus \g_2$. Taking the projection on each factor, we obtain two self-adjoint ad-invariant maps $\pi_1,\pi_2\colon \g\to\g$. By the argument of Remark~\ref{remark:tracess_not_rigid}, neither $\pi_1$ nor $\pi_2$ have the form $D+D^*$; since they are linearly independent, $g$ is not \weaklyrigid.
\end{remark}

According to Corollary~\ref{cor:phihasoneortwoeigenvalues}, on an irreducible real Lie algebra, self-adjoint ad-invariant maps are mainly of three types: the identity, a complex structure, or a nilpotent endomorphism. The following examples show, for each one of these three types of self-adjoint ad-invariant maps, a Lie algebra endowed with an ad-invariant metric which fails to be \rigid due to such a map not being the self-adjoint part of a derivation.

\begin{example}\label{ex:simple}
Any simple Lie algebra $\g$ has an ad-invariant metric defined by the Killing form; since every ad-invariant map $\g\to\g$ is a multiple of the identity, $\g$ is \weaklyrigid. On the other hand, all derivations are inner, hence skew-adjoint and thus traceless, so $\g$ is not \rigid because of Remark~\ref{remark:tracess_not_rigid}.

Similarly, one sees that the Killing form on a nonsimple, semisimple Lie algebra is ad-invariant and not \weaklyrigid. Indeed, the projection on any simple ideal is an ad-invariant and self-adjoint linear map, but cannot be written in the form $D+D^*+\lambda \id$.
\end{example}

\begin{example}\label{ex:SolvableTracelessDerivations}
Consider the $11$-dimensional solvable Lie algebra $\g$ with structure constants
\begin{multline}
\label{eqn:firstinsalamonnotation}
(e^{3,10}, e^{2,10}, 0, -e^{4,10},
e^{39} +e^{9,10}, e^{34}-e^{6,10},\\
-e^{19}-e^{24}-e^{5,10}, e^{23}+e^{8,10},
e^{13}-e^{1,10}, 0, e^{19}-e^{26}-e^{35}+e^{48}).
\end{multline}
This notation, that will be used throughout the paper, means that $\mg^*$ has a basis $\{e^1, \ldots, e^{11}\}$ with the Chevalley-Eilenberg differential of each element determined by~\eqref{eqn:firstinsalamonnotation}, i.e. $de^1=e^3\wedge e^{10}$ and so on.

The Lie algebra $\mg$ admits  the ad-invariant metric $g$ given by:
\[e^{1}\odot e^{5}+e^{2}\odot e^{6}+e^{3}\odot e^{7}+e^{4}\odot e^{8}+e^{9}\otimes e^{9}+e^{10}\odot e^{11}.\]
A straightforward computation shows that $g$  is not \rigid, as all derivations are traceless, but it is \weaklyrigid. In particular, it is irreducible by Remark~\ref{rk:weakimpliesirred}. We also note that $\g'=[\g,\g']$, so this is an example of a nonnilpotent solvable Lie algebra with a \weaklyrigid ad-invariant metric.
\end{example}

\begin{example}\label{ex:CharacteristicallyNilpotentDIM12}
Consider the $12$-dimensional nilpotent Lie algebra $\mg$ given by:
\begin{multline*}
(0,0,e^{14},e^{12},e^{47}+e^{28}+e^{39}+e^{10,11},e^{18}+e^{4,10}-e^{3,11}-e^{9,11},\\
- e^{19}-e^{2,11}+e^{4,11},-e^{17}+e^{2,10}-e^{3,11},e^{13}-e^{2,11},e^{24}+e^{1,11},0,-e^{1,10}-e^{23}-e^{29}-e^{34}).
\end{multline*}
This is a \emph{characteristically nilpotent} Lie algebra, meaning that all derivations are nilpotent (see, for instance, \cite{AC01}).

It admits the ad-invariant metric
\[g=-e^1\odot e^5+e^2\odot e^6-e^3\odot e^7-e^4\odot e^8-e^9\otimes e^{9}+e^{10}\otimes e^{10}+e^{11}\odot e^{12}.\]
A straightforward computation shows that the only ad-invariant self-adjoint endomorphisms of $(\mg,g)$ are the multiples of the identity. Thus, $\g$ is \weaklyrigid but not \rigid by Remark~\ref{remark:tracess_not_rigid}.

Note that this Lie algebra is irreducible by Remark~\ref{rk:weakimpliesirred}.
\end{example}

\begin{example}
Take the complexification $\g^\C$ of the Lie algebra $\g$ of Example~\ref{ex:CharacteristicallyNilpotentDIM12}, and consider its underlying real Lie algebra $(\g^\C)^\R$.

One can check that $(\g^\C)^\R$ is irreducible and all its derivations are traceless. Thus, the almost complex structure, which is ad-invariant and self-adjoint, cannot be written as $D+D^*$, for any derivation $D$. Hence, this Lie algebra is neither \rigid nor \weaklyrigid.
\end{example}

\begin{example}
Given a Lie algebra $\g$, one can define a Lie algebra $\g \ltimes \g^*$, with $\g$ acting on $\g^*$ by the coadjoint representation, called the cotangent Lie algebra. It is well known that cotangent Lie algebras admit a canonical ad-invariant metric $g$ induced  by the natural pairing between $\g$ and $\g^*$ (see Section~\ref{sec:Cotangents}).

Fix $\g$ as in Example~\ref{ex:CharacteristicallyNilpotentDIM12}, and let $\phi\colon\g\ltimes\g^*\to\g\ltimes\g^*$ be the map defined by:
\[\phi(v,\alpha)=(0,g(v,\cdot)), \quad v\in\g,\ \alpha\in \g^*.\]
A straightforward computation shows that $\phi$ is self-adjoint with respect to the metric $g$, ad-invariant and nilpotent, since $\im\phi=\ker\phi$. Moreover, straightforward calculations show that there is no derivation $D$ of $\g\ltimes\g^*$ such that $\phi=D+D^*$, hence $\g\ltimes\g^*$ is not \rigid nor \weaklyrigid.
\end{example}

The following elementary fact will be used repeatedly.
\begin{lemma}
\label{lemma:derivationcomposedadinvariant}
Let $\g$ be a Lie algebra; let $\phi\colon \g\to\g$ be ad-invariant, and let $D\colon\g\to\g$ be a derivation. Then  $\phi\circ D$ is a derivation.
\end{lemma}
\begin{proof}
For any $X,Y\in\mg$, we have
\[\phi(D[X,Y])=\phi([DX,Y]+[X,DY])=[\phi(DX),Y]+[X,\phi(DY)].\qedhere\]
\end{proof}

For every ad-invariant metric $g$, define
\[S_g=\{\phi\colon\g\to\g\st \phi \text{ ad-invariant}, \phi=\phi^{*_g}\}, \quad D_g=\{D+D^{*_g}\st D\in\Der\g\}.\]
We know from Lemma~\ref{lemma:symmetrized_derivation_ad_invariant} that $D_g\subset S_g$; we claim that the dimensions of $D_g$ and $S_g$ are independent of the ad-invariant metric $g$; in particular the codimension of $D_g$ in $S_g$ is constant.
\begin{lemma}
\label{lemma:ShDhconstant}
Let $g,h$ be ad-invariant metrics on $\g$ related by $h=g(\tau\cdot,\cdot)$. Then
\begin{equation}
\label{eqn:ShDh}
S_h=\{\tau^{-1}\phi\st \phi \in S_g\}, \quad D_h=\{\tau^{-1}\phi\st \phi\in D_g\}.
\end{equation}
In particular, $\dim S_g=\dim S_h$, $\dim D_g=\dim D_h$.
\end{lemma}
\begin{proof}
By construction, $\tau=\tau^{*_g}$. For any endomorphism $f$, we have
\[h(fX, Y)=g(\tau f X,Y)=g(fX,\tau Y)=g(X,f^*\tau Y)=h(X,\tau^{-1}f^*\tau Y).\]
Hence, the adjoint relative to $h$ is $f^{*_h}=\tau^{-1}f^*\tau$, or equivalently $\tau\phi^{*_h}=(\tau\phi)^{*_g}$.

Thus, if $\phi\in S_h$, we have $\tau\phi=(\tau\phi)^*$, i.e. $\tau\phi\in S_h$. Similarly, given $D+D^{*_h}\in D_h$, we have that
\[\tau(D+D^{*_h})=\tau D + (\tau D)^{*_g} \in D_g\]
since $\tau D$ is a derivation by Lemma~\ref{lemma:derivationcomposedadinvariant}. Reversing the roles of $g$ and $h$ gives the opposite inclusions.
\end{proof}

Since ad-invariant metrics $g$ are \rigid when $S_g=D_g$, we obtain:
\begin{proposition}
\label{prop:rigid_independent}
Given a Lie algebra $\g$ and an ad-invariant metric $g$, then $g$ is \rigid  if and only if every ad-invariant metric on $\mg$ is \rigid.
\end{proposition}

From now on and in view of the previous result, we will say that a Lie algebra $\mg$ is \rigid if one (and thus every) ad-invariant metric on $\mg$ is \rigid.

Recall that any Lie algebra with an ad-invariant metric is the orthogonal direct sum of irreducible ideals.
\begin{theorem}
\label{thm:rigidimpliesuniqueness}
Let $\g$ be a Lie algebra over $\K=\R,\C$ with an ad-invariant metric $g$. Then
\begin{enumerate}[label=(\roman*)]
\item \label{item:uniquecomplex} for $\K=\C$, $g$ is \rigid if and only if  every ad-invariant metric on $\g$ has the form $ g(\psi\cdot,\psi\cdot)$ for some $\psi\in\Aut(\g)$;
\item \label{item:uniquereal} for $\K=\R$ and $\g$ irreducible, $g$ is \rigid
if and only if every ad-invariant metric on $\g$ has the form $\pm g(\psi\cdot,\psi\cdot)$ for some $\psi\in\Aut(\g)$;
\item \label{item:almostuniquereal} for $\K=\R$ and $\g$ reducible, fix a decomposition $\g=\g_1\oplus\dots \oplus \g_k$ into irreducible orthogonal ideals, and decompose accordingly $g$ as $g_1+\dotsc + g_k$; then $g$ is \rigid if and only if every ad-invariant metric on $\g$ has the form
\[h(\psi\cdot,\psi\cdot), \quad \psi\in\Aut(\g),\; h=\pm g_1\pm \dots \pm g_{k}.\]
\end{enumerate}
\end{theorem}
\begin{proof}
First we shall prove the ``if'' part of each item.
Let $\mathcal{K}$ denote the set of ad-invariant metrics on $\g$, and assume that $\Aut(\g)$ acts on $\mathcal{K}$ with finitely many orbits. Then at least one orbit is open. Let $h$ be an ad-invariant metric in the open orbit and let $\phi$ be an ad-invariant self-adjoint map. Consider the family of ad-invariant metrics
\[h_t(X,Y)=h((t\phi+\id)X,Y).\]
For small $t$, $h_t$ is in the $\Aut(\mg)$-orbit of $h$, so there exists a curve of Lie algebra automorphisms $\psi(t)$ such that $\id+t\phi=\psi(t)^*\psi(t)$; we can assume $\psi(0)=\id$. Therefore, taking the derivative at $t=0$, we get
\[\phi = D^*+D,\]
where $D=\psi'(0)$ is a derivation. This shows that $h$ is solitary; by Proposition~\ref{prop:rigid_independent}, every metric is solitary.

Conversely, assume that $g$ is solitary, and write $\g$ as an orthogonal sum of irreducible ideals, $\g=\g_1\oplus\dots \oplus \g_k$. If $g'$ is another ad-invariant metric, determining an analogous decomposition $\g=\g'_1\oplus \dots \oplus \g'_{k'}$, the Lie algebra version of the  Krull-Schmidt theorem (see \cite{FisherGrayHydon}) implies that $k'=k$ and up to permuting the factors we can assume $\g'_i\cong\g_i$; this means that some automorphism of $\g$ maps one decomposition into the other. In other words,  $\g=\g_1\oplus\dots \oplus \g_k$ is also an orthogonal sum relative to some metric in the $\Aut(\g)$-orbit of $g'$. Thus, it suffices to prove uniqueness of the metric (up to sign, in the real case) on each irreducible factor $\g_i$.

Assume from now on that $\g$ is irreducible and solitary. Let $\Z_2$ act on $\mathcal{K}$ by an overall sign change. We must prove that $\Aut(\g)$ acts transitively on $\mathcal{K}$ (respectively, $\mathcal{K}/\Z_2$). Since  $\Aut(\g)$-orbits in $\mathcal{K}$ are open, it suffices to prove that $\mathcal{K}$ (respectively $\mathcal{K}/\Z_2$) is connected.

Let $h$ be an ad-invariant metric. Then $h(X,Y)=g(\phi X,Y)$ for some ad-invariant, self-adjoint $\phi$. Since we assume $\g$ to be irreducible, we know from Proposition~\ref{prop:eigenvaluesofphi} that $\phi$ cannot have more than one real eigenvalue; notice that $0$ is not an eigenvalue because $\phi$ is invertible. Up to changing the sign of $h$, we can assume that there are no negative eigenvalues.

Since
\[\det (t\id + (1-t)\phi)=0 \iff \det \left(\phi + \frac{t}{1-t}\id\right)=0,\]
we have that convex combinations of $\id$ and $\phi$ are nondegenerate. Thus, we obtain a curve
\[h_t(X,Y)=g((t\id + (1-t)\phi)X,Y)\]
connecting $g$ and $h$ in $\mathcal{K}$. Notice that in order to achieve this we might have had to change the sign of $h$, but this is sufficient to prove connectedness of $\mathcal{K}/\Z_2$.

In the complex case, writing $i\id=D+D^*$, we see that the metrics $e^{i\theta}h$ are in the same $\Aut(\g)$-orbit.
\end{proof}

The analogous result for Lie algebras with a weakly solitary ad-invariant metrics is complicated by the fact that the Lie algebra may also admit an ad-invariant metric which is not weakly solitary, as in the following:
\begin{example}
\label{example:Tstarsu2}
Consider the cotangent of $\su(2)$, namely the Lie algebra
\[(e^{23},- e^{13},e^{12},e^{26}-e^{35},e^{34}-e^{16},-e^{24}+e^{15}).\]
This Lie algebra admits two inequivalent ad-invariant metrics, namely
\[g=e^1\odot e^4+e^2\odot e^5+e^3\odot e^6, \qquad
 h=g+e^1\otimes e^1+e^2\otimes e^2+e^3\otimes e^3;\]
every ad-invariant metric is a linear combination of $g$ and $h$. This implies that $S_{g}$ is two-dimensional, and we can check that $D_{g}=\Span{\id}$. However $\id$ is not in $D_{h}$; it follows that $h$ is weakly solitary, whilst $g$ is not weakly solitary.
\end{example}

The next results will consider the action of $\Aut(\g)\times \K^*$ on the space of ad-invariant metrics. Namely, for an ad-invariant metric $g$ on $\mg$, the action is given by $kg(\psi\cdot,\psi\cdot)$ with $\psi\in \Aut(\g)$ and $k\in\K^*$.

\begin{lemma}
\label{lemma:twoorbits}
Let $\g$ be irreducible and let $g_1$ be a weakly solitary ad-invariant metric. If $g_2$ is an ad-invariant metric on $\g$ which is not of the form $kg_1(\psi\cdot,\psi\cdot)$ for an automorphism $\psi$ and $k\in\K$, then $g_2$ is not weakly solitary. Moreover, $g_2$ is unique up to a change of sign and an automorphism.
\end{lemma}
\begin{proof}
If $g_1$ is solitary the statement follows from Theorem~\ref{thm:rigidimpliesuniqueness}; suppose otherwise. Then for every ad-invariant metric $g$,  the codimension of $D_g$ in $S_g$ is one because of Lemma~\ref{lemma:ShDhconstant}. Therefore, since $\mg$ is not solitary, $g$ is weakly solitary if and only if $\id\notin D_g$.

Let $\mathcal{K}$ be the space of ad-invariant metrics on $\mg$, and let
\[\mathcal{K}_0=\{g\in\mathcal{K}\st \id\in D_g\}\]
be the set of those which are not weakly solitary. Arguing as in Theorem~\ref{thm:rigidimpliesuniqueness}, we see that weakly solitary metrics are exactly the elements of $\mathcal{K}$ whose $\Aut(\g)\times\K^*$-orbit is open.

We need to show that $\Aut(\g)\times\K^*$ acts transitively on $\mathcal{K}\smallsetminus\mathcal{K}_0$ and that if  $\mathcal{K}_0$ is nonempty then $\Aut(\g)\times\Z_2$ acts transitively on it. Notice that in the case where $\mathcal{K}_0$ is empty, the fact that $\Aut(\g)\times\K^*$ acts transitively on $\mathcal{K}$ follows from the argument of Theorem~\ref{thm:rigidimpliesuniqueness}.

Now assume that $g$ and $h=g(\tau\cdot,\cdot)$ are in $\mathcal{K}_0$, i.e. such that $D_g,D_h$ contain the identity. Write $*$ instead of $*_g$.

By~\eqref{eqn:ShDh}, since $\id\in D_h$, $\tau$ belongs to $D_g$. Conversely, for every $\phi\in D_g$, we have that $h_\phi=g(\phi\cdot,\cdot)$ is such that $D_{h_{\phi}}$ contains the identity, i.e. it is in $\mathcal{K}_0$.

Now consider the affine line $t\mapsto \phi_t=t\id+(1-t)\tau$ in $D_g$. By the above observation, we have that $D_{h_{\phi_t}}$ contains the identity. Arguing as in Theorem~\ref{thm:rigidimpliesuniqueness}, up to changing the sign of $h$, we can assume that $h_{\phi_t}$ is invertible for all $t\in [0,1]$. This shows that there is a path inside $\mathcal{K}_0$ connecting $g$ and $\pm h$, i.e. $\mathcal{K}_0/\Z_2$ is connected.

Now write $S_g=D_g+\Span{\psi}$, and consider the affine line $\phi_t=\id+t\psi$. Suppose that for some $t$ we have $\id\in D_{h_{\phi_t}}$; then $\phi_t\in D_g$ by~\eqref{eqn:ShDh}, which implies $t=0$. Since $\psi\notin D_g$, $\{h_{\phi_t}\}$ is a section for the action of $\Aut(\g)$, and thus non-weakly-solitary $\Aut(\g)$-orbits are isolated. Since they form a connected space up to sign, there is either only one such orbit, or two related by a sign change. By a dimensional count, $\Aut(\g)$ acts with open orbits on $\mathcal{K}_0$, so it acts transitively on $\mathcal{K}_0/\Z_2$.

It remains to show that $\mathcal{K}_0/\Z_2$ does not disconnect $\mathcal{K}/\Z_2$. In order to see this, it suffices to show that some $\phi_t$ is in the same component as $\phi_{-t}$ (equivalently,  $-\phi_{-t}$). But this follows from taking the affine line
\[s\mapsto s\phi_t - (1-s)\phi_{-t}=s(\id+t\phi)-(1-s)(\id-t\phi)=(2s-1)\id+t\psi,\]
which is never in $D_g$ for $t\neq0$. This means that one can go from $\phi_t$ to $-\phi_{-t}$ without intersecting $\mathcal{K}_0$. Since $\Aut(\g)\times\K^*$ acts with open orbits on the connected space $(\mathcal{K}\smallsetminus\mathcal{K}_0)/\Z_2$, we see that the action is transitive.
\end{proof}

\begin{remark}
Given an irreducible Lie algebra with an ad-invariant metric $g$ such that $\dim S_g=\dim D_g+1$, the proof of Lemma~\ref{lemma:twoorbits}  shows that there exists a weakly solitary ad-invariant metric.
\end{remark}

\begin{proposition}
\label{prop:weaksol}
Let $\g$ be an irreducible Lie algebra over $\K=\R,\C$ with an ad-invariant metric $g$. Relative to the action of $\Aut(\g)\times\K^*$ on the set $\mathcal{K}$ of ad-invariant metrics on $\g$, we have:
\begin{enumerate}[label=(\roman*)]
\item $g$ is weakly solitary if and only if its orbit is open;
\item\label{it:trfree} if all derivations of $\g$ are traceless and $g$ is weakly solitary, $\mathcal{K}$ contains exactly one orbit;
\item if $\g$ admits a derivation $D$ with nonzero trace and $g$ is weakly solitary, $\mathcal{K}$ contains exactly two orbits, one containing $g$, the other containing $h=g((D+D^{*_g})\cdot,\cdot)$, which is not weakly solitary.
\end{enumerate}
\end{proposition}
\begin{proof}
The first item is obvious.

For the second item, observe that an ad-invariant metric $h$ is in $\mathcal{K}_0$ if and only if $\id=D+D^{*_h}$, which implies $\Tr D\neq0$.

For the third item, write $\tau=D+D^{*_g}$; then $\Tr \tau=2\Tr D\neq0$, which by Corollary~\ref{cor:phihasoneortwoeigenvalues} implies that $\tau$ is invertible, so $h=g(\tau\cdot,\cdot)$ is a well-defined ad-invariant metric. Then $D_h=\tau^{-1}D_g$ contains $\tau^{-1}\tau=\id$, so $h$ is not weakly solitary.
\end{proof}

\begin{remark}\label{rem:ComparsionWithThm34MeRe}
Let $g$ be an ad-invariant metric on a Lie algebra $\mg$ such that every self-adjoint ad-invariant map of $\mg$ can be written as a linear combination of the identity and an endomorphism satisfying~\eqref{eqn:phiWZ}. According to \cite[Theorem 3.4]{MeRe93}, the ad-invariant metric is unique up to automorphisms and rescaling. We can recover this result as follows. Let $W_g$ be the space of self-adjoint maps $\g\to\g$ satisfying~\eqref{eqn:phiWZ}. Then $W_g\subset D_g\subset S_g
=W_g+\Span{\id}$. Therefore, either $W_g=D_g$, implying that derivations are traceless and~\ref{it:trfree} in Proposition~\ref{prop:weaksol} applies, or $D_g=S_g$ and Theorem~\ref{thm:rigidimpliesuniqueness} implies uniqueness of the metric.
\end{remark}

From Proposition~\ref{pro:decomposable} we know that a reducible Lie  algebra $\g=\g_1\oplus\g_2$ admits an ad-invariant metric if and only if each factor $\mg_1$, $\mg_2$ admits such a metric. The next result shows that the \rigid condition for $\mg$ also reduces to the same condition on each factor.
\begin{proposition}
\label{prop:decomposablesolitary}
Let $\g=\g_1\oplus\g_2$ be a reducible Lie algebra. Then $\g$ is \rigid if and only if $\g_1$ and $\g_2$ are \rigid.
\end{proposition}
\begin{proof}Take a linear map $f\colon\g\to\g$, and decompose it as
\[f=\begin{pmatrix} f_1 & f_3 \\ f_2 & f_4\end{pmatrix},\]
with respect to the decomposition of $\mg$.
For $X,Y\in\g_1$, we have
\[f[X,Y]-[fX,Y]=(f_1[X,Y]-[f_1X,Y]) + f_2[X,Y]\]
and for $X_1\in\g_1$, $X_2\in\g_2$ we have
\[f[X_1,X_2]-[fX_1,X_2]=-[f_2X_1,X_2].\]
This shows that $f$ is ad-invariant if and only if $f_1\colon\g_1\to\g_1$ is ad-invariant, $f_4\colon\g_2\to\g_2$ is ad-invariant, and the off-diagonal component
$\begin{pmatrix} 0 & f_3 \\ f_2 & 0\end{pmatrix}$
satisfies~\eqref{eqn:phiWZ}.

A similar fact holds for derivations.

Now fix on $\g$ an ad-invariant metric which is the sum of the two ad-invariant metrics on $\g_1$ and $\g_2$. It follows that a self-adjoint ad-invariant $\phi\colon\g\to\g$ is the self-adjoint part of a derivation if and only if $\phi_1$ and $\phi_4$ are the self-adjoint part of a derivation, so $\g$ is \rigid if and only if the components are \rigid.
\end{proof}

\section{The real, the complex, the solitary}\label{sec:complexification}
Let $\g$ be a real Lie algebra with an ad-invariant metric $g$. Then its complexification $\g^{\C}$ admits a natural ad-invariant complex scalar product $\tilde g$, obtained by extending scalars to $\C$. We will denote the adjoint relative to $g$ by $*$ and the adjoint relative to $\tilde g$ by $\tilde *$.

It turns out that $\g$ is solitary if and only if so is $\g^\C$. This follows from the following elementary fact:
\begin{lemma}
\label{lemma:elementary}
Let $\g$ be a Lie algebra and $\g^\C$ be its complexification. Then any $\C$-linear map $f\colon\g^\C\to\g^\C$ can be written in a unique way as $A^\C+iB^\C$, where $A,B\colon\g\to\g$ are $\R$-linear. Moreover:
\begin{enumerate}[label=(\roman*)]
\item $f$ is ad-invariant if and only if $A$ and $B$ are ad-invariant;
\item $f$ is a derivation if and only if $A$ and $B$ are derivations;
\item $f^{\tilde *}=(A^*)^\C+i(B^*)^\C$.
 \end{enumerate}
\end{lemma}
\begin{proof}
Let $e_1,\dotsc, e_n$ be a basis of $\g$; then $e_1,\dotsc, e_n$ can be seen as a basis of $\g^\C$ over $\C$.

A linear map $f\colon\g^\C\to \g^\C$ can be identified with a complex matrix $(f_{ij})$; the real linear maps $A$ and $B$ correspond to the real matrices $(\Re f_{ij})$, $(\Im f_{ij})$, where $\Re $ and $\Im $ denote real and imaginary part, respectively.

Since the structure constants in the basis $e_1,\dotsc, e_n$ are real, it is clear that the equations in $(f_{ij})$ characterizing ad-invariant maps have real coefficients, so $f$ is ad-invariant if and only if so are $A$ and $B$. The case of derivations is similar.

The last item follows from the fact that the scalar products $g$ and $\tilde g$ are represented by the same real matrix in the basis $e_1,\dotsc, e_n$.
\end{proof}

\begin{proposition}
\label{prop:solitaryiffcomplexsolitary}
Let $\g$ be a real Lie algebra with an ad-invariant metric $g$. Then $\g$ is \rigid if and only if the complexification $g^\C$ is \rigid on $\g^\C$.
\end{proposition}
\begin{proof}
Let $g$ be a \rigid ad-invariant metric. Let $\psi\colon\g^\C\to\g^\C$ be ad-invariant and self-adjoint relative to $\tilde g$, and write $\psi=A^\C+iB^\C$ as in Lemma~\ref{lemma:elementary}; then $A,B\colon\g\to\g$ are ad-invariant and self-adjoint. Since $g$ is \rigid, we can write $A=D+D^*$, $B=E+E^*$, where $E$ and $D$ are derivations of $\g$. Again by Lemma~\ref{lemma:elementary}, we have
\[(D^\C+iE^\C)^{\tilde *}=(D^*)^\C+i(E^*)^\C,\]
so
\[\psi=(D^\C+iE^\C) + (D^\C+iE^\C)^{\tilde *}\]
and $\tilde g$ is \rigid.

Conversely, suppose that $g^\C$ is \rigid, and let $\phi\colon\g\to\g$ be ad-invariant and self-adjoint. We have
$\phi^\C=D+D^{\tilde *}$, with $D$ a derivation of $\g^\C$. By Lemma~\ref{lemma:elementary}, we can write $D=A^\C+iB^\C$, with $A,B$ derivations of $\g$, and
\[\phi^\C=A^\C+iB^\C + (A^\C+iB^\C)^{\tilde *}=\tilde A^\C + i\tilde B^\C, \quad \text{with } \tilde A= A+A^*,\; \tilde B=B+B^*.\]
By the uniqueness of the decomposition $\phi^\C=\tilde A^\C + i\tilde B^\C$, this implies that $A+A^*=\phi$ and $B+B^*=0$, and therefore $g$ is \rigid.
\end{proof}

Similarly, given a complex Lie algebra $\g$ with an ad-invariant metric $g$, the underlying real Lie algebra $\g^\R$ has an ad-invariant metric
\[g^\R\colon \g^\R\times \g^\R\to\R, \quad g^\R(X,Y)=\Re g(X,Y);\]
the complex structure $J\colon\g^\R\to \g^\R$ determined by multiplication by $i$ is self-adjoint and ad-invariant. Notice that $g^\R$ is nondegenerate because $g(X,Y)\neq 0$ implies that one of $\Re g(X,Y)$, $\Im g(X,Y)=\Re g(X,-JY)$ is nonzero.

Moreover, we can consider the conjugate Lie algebra $\overline{\g}$, which corresponds to the same underlying real Lie algebra $\g^\R$ with the sign of the complex structure $J$ reversed. We have an ad-invariant metric
\[\overline{g}\colon \overline{\g}\times \overline{\g}\to\C, \quad \overline{g}(X,Y)=\overline{g(X,Y)}.\]
\begin{lemma}
\label{lemma:trivial}
Let $\g$ be a complex Lie algebra with an ad-invariant metric $g$. Then the conjugate Lie algebra $\overline{\g}$ with the ad-invariant metric $\overline{g}$ is \rigid if and only if $\g$ is \rigid.
\end{lemma}
\begin{proof}
A linear map $\g^\R\to\g^\R$ defines a $\C$-linear map $\g\to\g$ if it commutes with $J$; this is equivalent to commuting with $-J$, so linear maps $\g\to\g$ can be identified with linear maps $\overline{\g}\to\overline{\g}$, and being ad-invariant or a derivation for $\g$ is equivalent to being ad-invariant or a derivation for $\overline{\g}$.

Likewise, the $g$-adjoint of an operator $f\colon\g\to\g$ can be identified with the $\overline{g}$-adjoint of the corresponding operator $f\colon\overline{\g}\to\overline{\g}$. The statement follows.
\end{proof}

\begin{proposition}
\label{prop:solitaryiffrealsolitary}
Let $\g$ be a complex Lie algebra with an ad-invariant metric $g$, and let $\g^\R$ be the underlying real Lie algebra, with the ad-invariant metric $g^\R=\Re g$. Then $\g$ is solitary if and only if $\g^\R$ is solitary.
\end{proposition}
\begin{proof}
Let $\mh=(\g^\R)^\C$ be the complexification of $\g^\R$. The complex structure $J\colon\g^\R\to\g^\R$, extended $\C$-linearly to $\mh$, determines a decomposition into $\pm i$ eigenspaces
\[\mh=\mh^{1,0}\oplus\mh^{0,1},\]
where $\mh^{1,0}$ and $\mh^{0,1}$ are commuting ideals. Then $\mh$ carries an ad-invariant metric obtained by complexifying $g^\R$; $\mh^{1,0}$ and $\mh^{0,1}$ are orthogonal.

As complex Lie algebras, $\mh^{1,0}$ is isomorphic to $\g$ and $\mh^{0,1}$ is isomorphic to $\overline{\g}$. By Lemma~\ref{lemma:trivial} and Proposition~\ref{prop:decomposablesolitary}, $\mg$ is \rigid if and only if $\mh$ is \rigid. On the other hand, by Proposition~\ref{prop:solitaryiffcomplexsolitary} $\mh$ is \rigid if and only if so is $\g^\R$, and we obtain the asserted equivalence.
\end{proof}

\begin{remark}
As observed before, given a real Lie algebra $\g$ endowed with an ad-invariant metric $g$, then there exists an ad-invariant complex structure $J^2$ if and only if $\tilde\mg=(\g,J)$ is a complex Lie algebra and $g$ is the real part of a $\C$-bilinear ad-invariant metric $h$ on $\tilde\mg$.

Moreover, using the same argument as in the proof of Proposition~\ref{prop:solitaryiffrealsolitary}, we can show that $\id=\tilde D+\tilde D^{*_{h}}$ for some $\tilde D\in \Der(\tilde\mg)$ if and only if $J=D+D^{*_g}$ for some $D\in \Der(\mg)$.
\end{remark}

Whilst for irreducible Lie algebras the \rigid condition is equivalent to uniqueness of the ad-invariant metric up to sign, we cannot conclude that a real Lie algebra $\g$ has a unique ad-invariant metric (up to sign) if and only if $\g^\C$ has a unique ad-invariant metric.

\begin{example}
\label{ex:FF}
The real Lie algebra
\[\lie{f}:\quad (0,0,12,13,23)\]
has an ad-invariant metric $g_\lie{f}=e^1\odot e^5+e^2\odot e^4+e^3\otimes e^3$, which can easily be checked to be solitary. Notice that $g_\lie{f}$ and  $-g_\lie{f}$ have different signatures, so it is not possible to obtain one from the other by an automorphism.

If we take the direct sum of two copies of $\lie{f}$, call them $\lie{f}_1$ and $\lie{f}_2$, we have  that the Lie algebra $\mg=\lie{f}_1\oplus \lie{f}_2$ has three ad-invariant metrics of different signatures, namely
\[g_{\lie{f}_1}+g_{\lie{f}_2},\ g_{\lie{f}_1}-g_{\lie{f}_2},\ -g_{\lie{f}_1}-g_{\lie{f}_2}.\]
Let $\Ann\g' \subset\g^*$ be the annihilator of the commutator $\g'$, we see by an explicit computation that the space of ad-invariant symmetric tensors is $S^2(\Ann\g')+\Span{g_{\lie{f}_1},g_{\lie{f}_2}}$, which contains exactly three orbits under the group of automorphisms.

However, the same computation over $\C$ shows that $\lie{f}_1^\C\oplus \lie{f}_2^\C$ has a unique metric up to automorphisms.
\end{example}

\begin{example}
The real $10$-dimensional Lie algebra underlying $\lie{f}^\C$ can be written as
\begin{gather*}
\g:\quad (0,0,0,0,-e^{13}+e^{24},e^{12}+e^{34},e^{51}+e^{46},e^{45}-e^{61},e^{62}+e^{53},e^{25}-e^{36});
\end{gather*}
it admits the neutral ad-invariant metric
\[g=e^1\odot e^{10}+e^2\odot e^7+e^3\odot e^8+e^4\odot e^9+e^5\odot e^6.\]
It is solitary and irreducible, thus the metric is unique up to a change of sign.

This example should be contrasted with Example~\ref{ex:FF}; notice that the complexification of $\g$ is isomorphic to $\lie{f}_1^\C\oplus \lie{f}_2^\C$.
\end{example}

\section{Gradings adapted to the metric}\label{sec:Canonical}
In this section we introduce a class of gradings on a metric Lie algebra which is ``adapted'' to the metric in some sense. The section is divided in three parts. In the first, we introduce the relevant class of gradings and show its use in proving that an ad-invariant metric is solitary. In the second, we introduce a canonical grading on any metric Lie algebra, without assuming that the metric is ad-invariant; the grading is determined by the eigenspaces of a canonical semisimple derivation. The construction is adapted from \cite{Nikolayevsky}; therefore, this derivation will be called the \emph{metric Nikolayevsky derivation}. In the last part, we establish a sufficient condition for an ad-invariant metric to be solitary.

\subsection{Gradings and the solitary condition}
Recall that a \emph{grading} of a Lie algebra $\g$ is a decomposition $\mg=\bigoplus_{i} \mb_i$ such that
\begin{equation}\label{eq:gradb}
[\mb_i,\mb_j]\subseteq \mb_{i+j};
\end{equation}
we will mostly consider gradings over $\Z$, but we will occasionally allow the indices $i$ to be complex numbers.

\begin{remark}\label{rem:nilpgrad}
If $\mg$ admits a grading over $\N_0$, then $\mb_0$ is a subalgebra of $\mg$, $\mn:= \bigoplus_{i\in \N} \mb_i$ is a nilpotent ideal and $\mg$ can be written as a semidirect product $\mg=\mb_0\ltimes \mn$. Even more, when the grading is over $\N$ (i.e. $\mb_0=0$), $\mg$ is nilpotent.
\end{remark}

A linear map $f\colon\g\to\g$ is said to be of degree $k$ with respect to the grading $\mg=\bigoplus_{i} \mb_i$ if the image of $\mb_i$ is contained in $\mb_{k+i}$ for all $i$. More generally, any linear map $f\colon\g\to\g$ can be decomposed as $f=\sum f_k$, where $f_k$ has degree $k$. The following elementary fact will be used repeatedly.
\begin{proposition}
\label{prop:thegradedpart}
Let $\g=\bigoplus \mb_i$ be a graded Lie algebra, and let $f\colon\g\to\g$ be a linear map. We get:
\begin{enumerate}[label=(\roman*)]
 \item if $f$ is ad-invariant, then $f_k$ is ad-invariant;
 \item if $f$ is a derivation, then $f_k$ is a derivation;
 \item given a metric on $\g$ satisfying
 \begin{equation}
 \label{eqn:co_gradation}
 \mb_i \perp \mb_j \, \text{ when } i+j\neq l
\end{equation}
for some fixed integer $l$, we have $(f_k)^*=(f^*)_k$. In particular, if $f$ is self-adjoint, so is each $f_k$.
\end{enumerate}
\end{proposition}
\begin{proof}
For each integer $k$, $f_k$ is defined by
\[f_k|_{\mb_i}=\pi_{k+i}\circ f|_{\mb_i}\]
with $\pi_{k+i}$ denoting the projection.

If $f$ is ad-invariant, then for every $X\in \mb_i$, $Y\in\mb_j$ we have
\[f_k([X,Y])=\pi_{k+i+j}f([X,Y])
 =\pi_{k+i+j}[f(X),Y] =[\pi_{k+i}fX,Y]=[f_k X,Y].
\]
Similarly, one sees that if $f$ is a derivation then $f_k$ is a derivation.

Given a metric $g$ satisfying~\eqref{eqn:co_gradation}, for $X\in \mb_i$, $Y\in\mb_j$ we have
\begin{align*}
g( f_kX,Y) &= g( \pi_{k+i}f X, Y) =
 g( f X, \pi_{l-k-i}Y) =\delta_{j,l-k-i}g( f X, Y)\\
 &= \delta_{j,l-k-i}g( X, f^* Y)
 =\delta_{j,l-k-i}g( X, \pi_{l-i} f^* Y) = g( X, \pi_{k}f^* Y) =g( X, (f^*)_k Y).\qedhere
\end{align*}
\end{proof}

\begin{lemma}
\label{lemma:gradationadaptedtophi}
Suppose that $\g$ is a Lie algebra over $\K=\R,\C$ with an ad-invariant metric and a grading
$\g=\bigoplus \mb_i$ satisfying~\eqref{eqn:co_gradation} for some fixed integer $l$. Suppose that $\phi$ is an ad-invariant self-adjoint operator of degree $k$. Then \[(l-k)\phi=D+D^*,\] where $D$ is the derivation $\sum_t t\phi|_{\mb_t}$.
\end{lemma}
\begin{proof}
Observe first that $N=\sum_t t\id_{\mb_t}$ is a derivation, so $D$ as in the statement is indeed $D=\phi N$ and  thus a derivation by Lemma~\ref{lemma:derivationcomposedadinvariant}. Since the grading satisfies~\eqref{eqn:co_gradation}, $N+N^*=l\id$ and thus
\[D+D^*=\phi N + N^*\phi = \phi N+(l\id-N)\phi = (\phi N-N\phi )+l\phi = (l-k)\phi.\qedhere\]
\end{proof}
\begin{lemma}
\label{lemma:uniquenesslemma}
Let $\g$ be a Lie algebra over $\K=\R,\C$ with an ad-invariant metric and a grading $\g=\bigoplus \mb_i$ satisfying~\eqref{eqn:co_gradation} for some fixed integer $l$. Suppose that  every ad-invariant, self-adjoint $\phi\colon\g\to\g$ of degree $l$ has the form $\phi=D+D^*$, where $D$ is a derivation of degree $l$. Then the metric
 is \rigid.
\end{lemma}
\begin{proof}
Suppose that $\phi$ is an ad-invariant self-adjoint operator; we must show that $\phi=D+D^*$ for some derivation $D$. By Proposition~\ref{prop:thegradedpart} we can decompose $\phi$ as the sum of graded ad-invariant self-adjoint operators $\phi_k$. We can therefore apply Lemma~\ref{lemma:gradationadaptedtophi} to each $\phi_k$ with $k\neq l$ and write each as $\phi_k=D_k+D_k^*$; in addition, we have $\phi_l=D_l+D_l^*$ by hypothesis. The statement is obtained by linearity.
\end{proof}

\subsection{The metric Nikolayevsky derivation}
Given a Lie algebra $\g$ over $\K=\R,\C$ with a metric $g$, not necessarily ad-invariant, we can define the group $\CO(\g,g)$ of conformal transformations, i.e. those linear maps $f\colon\mg\to\mg$ for which there is a constant $l$ such that
\[g(fX,fY)=lg(X,Y).\]
Its Lie algebra $\co(\g,g)$ (or simply $\co$) is the direct sum $\so(\g,g)\oplus\Span{\id}$.

Then $\CO(\g,g)\cap \Aut(\g)$ is an algebraic group; its Lie algebra is
\[\co(\g,g)\cap\Der(\mg) = \{f\in\Der( \g)\st \exists l\in \K\text{ s.t. }  g(fX,Y)+g(X,fY)=lg(X,Y)\}.\]

Given a grading $\g=\bigoplus \mb_t$, the derivation $N=\sum_t \id_{\mb_t}$ is in $\co$ if and only if~\eqref{eqn:co_gradation} holds.

\begin{lemma}
\label{lemma:semisimplepartofcoder}
Given a real metric Lie algebra $(\mg,g)$, every element $f$ of $\co(\g,g)\cap\Der(\mg)$ splits as
\[f=s_\R+s_{i\R}+n,\]
where all three summands are in $\co(\g,g)\cap\Der(\mg)$, $s_\R$ is semisimple with real eigenvalues, $s_{i\R}$ is semisimple with imaginary eigenvalues, and $n$ is nilpotent.
\end{lemma}
\begin{proof}
The decomposition into generalized eigenspaces of $f$ gives a grading $\g^\C=\bigoplus \mb_\lambda$. Suppose real numbers $\{\nu_\lambda\}$ are given such that $\nu_\lambda+\nu_\mu=\nu_{\lambda+\mu}$ whenever $\lambda,\mu,\lambda+\mu$ are eigenvalues. Then $\sum \nu_\lambda \id_{\mb_\lambda}$ is a derivation. In particular, this shows that
\[
s_\R=\sum (\Re \lambda) \id_{\mb_\lambda},\quad
 s_{i\R}=\sum i(\Im \lambda) \id_{\mb_\lambda}
\]
are derivations of $\g^\C$. Since they preserve the real structure, they are also derivations of $\g$.

In addition, if $g(fX,Y)+g(X,fY)=lg(X,Y)$, we have
\[
g((f-\lambda \id)^n X,Y)=g((f-\lambda v)^{n-1}X,((l-\lambda) \id-f)Y )
=\dots = g(X,((l-\lambda)\id-f)^n Y).
\]
Thus if $X\in \mb_i$, then $X$ is orthogonal to $\mb_j$ unless $i+j=l$.

This implies that $s_\R$ and $s_{i\R}$ are in $\co(\g,g)$ and thus so is the difference $n=f-s_\R-s_{i\R}$, which is nilpotent by construction.
\end{proof}
The same argument in the complex setting gives:
\begin{lemma}
\label{lemma:semisimplepartofcoderoverC}
Given a complex metric Lie algebra $(\mg,g)$, every element $f$ of $\co(\g,g)\cap\Der(\mg)$ splits as
\[f=s+n,\]
with $s$ semisimple, $n$ nilpotent and both are elements in $\co(\g,g)\cap\Der(\mg)$.
\end{lemma}

The construction of the metric Nikolayevsky derivation uses some simple combinatorics; we will make this explicit by introducing the following terminology
(along the lines of \cite{ContiRossi:Construction}). A \emph{symmetric diagram} is a digraph $\Delta$ with arrows labeled by nodes,
endowed with an involution of the set of nodes $\sigma$ such that:
\begin{enumerate}
 \item \label{item:N3}whenever $i\xrightarrow{j}k$ is an arrow, then $j\xrightarrow{i}k$ is also an arrow;
 \item \label{item:AD}whenever $i\xrightarrow{j}k$ is an arrow then $i\xrightarrow{\sigma_k}\sigma_j$ is also an arrow.
\end{enumerate}

The following is obvious:
\begin{lemma}\label{lem:prenice}
Given a constant $l$ and a finite set $\Lambda=\{\lambda_1,\dotsc, \lambda_n\}\subset\C$  such that $\Lambda=\{l-\lambda\st \lambda\in\Lambda\}$,
define a digraph $\Delta$ with nodes $\{1,\dotsc, n\}$ such that $i\xrightarrow{j}k$ is an arrow whenever $\lambda_i+\lambda_j=\lambda_k$, and
let $\sigma$ be the permutation satisfying $\lambda_i+\lambda_{\sigma(i)}=l$.  Then $(\Delta,\sigma)$ is a symmetric diagram.
\end{lemma}
Given a symmetric diagram $\Delta$ with nodes $\{1,\dotsc, n\}$, we can define a matrix $M_\Delta$ whose rows are $-e_i-e_j+e_k$ whenever $i\xrightarrow{j}k$ is an arrow. This is akin to the root matrix studied in the context of nice Lie algebras, but notice that we do not assume the indices $i,j,k$ to be distinct.
\begin{lemma}
\label{lemma:sprenice}
Let $(\Delta,\sigma)$ be a connected symmetric diagram. Then there is a constant $l$ such that any $X\in\ker M_\Delta$ satisfies
\[x_i+x_{\sigma_i}=l.\]
\end{lemma}
\begin{proof}
If $i\xrightarrow{j}k$ is an arrow, then $i\xrightarrow{\sigma_k}\sigma_j$ is an arrow. Thus, $M_\Delta$ contains both the rows $-e_i-e_j+e_k$ and $-e_i-e_{\sigma_k}+e_{\sigma_j}$, so $X\in\ker M_\Delta$ satisfies   $\langle X, e_j-e_k-e_{\sigma_k}+e_{\sigma_j}\rangle=0$, i.e.
\[x_j+x_{\sigma_j}=x_k+x_{\sigma_k}.\]
By connectedness, this shows that $x_i+x_{\sigma_i}$ is independent of $i$.
\end{proof}

Recall from \cite{Nikolayevsky} that every Lie algebra $\g$ admits a semisimple derivation $\tilde N$ such that
\[\Tr(\tilde N\psi)=\Tr\psi, \quad \text{ for every } \psi\in\Der(\g);\]
the derivation $\tilde N$ is unique up to automorphisms, and it is called the \textit{Nikolayevsky} (or \textit{pre-Einstein}) \textit{derivation}. Given a metric $g$ on $\mg$, in general, $\tilde N$ may not be in $\co(\mg,g)$. However, we can adapt the argument to construct a canonical derivation in $\co(\mg,g)$:
\begin{theorem}
\label{thm:metricnik}
Given a metric Lie algebra $(\mg,g)$ over $\K=\R,\C$, there is a semisimple derivation $N\in\co(\g,g)\cap\Der(\mg)$ such that
\begin{equation} \label{eqn:definingNik}
\Tr(N\psi)=\Tr\psi,\quad \text{ for every }\psi\in\co(\g,g)\cap\Der(\mg).
\end{equation}
This derivation is unique up to automorphisms preserving the conformal class of the metric and its eigenvalues are rational numbers.
\end{theorem}
\begin{proof}
The proof follows \cite[Section 3]{Nikolayevsky}. The Lie algebra $\co(\g,g)\cap\Der(\mg)$ is algebraic, so it admits a Levi decomposition $\lie{s} + \lie{r}$. The radical $\lie{r}$ is a solvable Lie algebra acting faithfully on $\g$; by Lie's theorem, we can assume that $\lie{r}$ consists of upper triangular matrices in $\gl(\g)$ (for $\K=\C$) or $\gl(\g^\C)$ (for $\K=\R)$. By Lemmas~\ref{lemma:semisimplepartofcoder},~\ref{lemma:semisimplepartofcoderoverC} we see that $\lie{r}$ splits as $\lie{a} + \lie{n}$, where $\lie{n}$ consists of the nilpotent elements of $\lie{r}$ and $\lie{a}$ is an abelian Lie algebra of semisimple elements.

For $\psi\in \lie{s}$, we have $\Tr\psi=0$, because $\lie{s}$ has no codimension one ideals. We also have $\Tr\psi=0$ for $\psi\in\lie{n}$. Notice that every element $N\in\lie{a}$ satisfies $\Tr (\psi N)=0$ for $\psi$ in $\lie{n}$ (because $\psi N$ is nilpotent) or  $\psi$ in $\lie{s}$ (because $\lie{s}+ \lie{a}$ is a representation of $\lie{s}$, and $(\phi, N)\mapsto \Tr(\phi N)$ is equivariant).

Therefore, we need to single out an $N\in \lie{a}$ such that $\Tr(\psi)=\Tr(N\psi)$ for all $\psi\in\lie{a}$. In the real case such an element always exists, because the scalar product
\[\lie{a}\times \lie{a} \to \K, \quad (\phi, N)\mapsto \Tr(\phi N)\]
is nondegenerate. In the complex case, we can fix a real structure on $\g$ as a vector space, thus obtaining a real structure on $\gl(\g)$ and a positive-definite hermitian product
\[\lie{a}\times \lie{a} \to \K, \quad (\psi, N)\mapsto \Tr(\psi\overline N).\]
We can summarize both the real and the complex case by saying that we have a unique $N\in\lie{a}$ such that $\Tr(\psi)=\Tr(\overline{N}\psi)$;  we will show that $N$ has rational eigenvalues, so that~\eqref{eqn:definingNik} is also satisfied.

Moreover $N$ is the unique element of $\lie{s} +\lie{a}$ that satisfies~\eqref{eqn:definingNik}, because the bilinear form $(N,\psi)\mapsto \Tr( N\psi)$ is nondegenerate on the semisimple Lie algebra $\lie{s}$ (indeed, the null space is a solvable ideal by Cartan's criterion, \cite[p.20]{Humphreys}, \cite[Section 3.4]{Samelson}). However, $\lie{s} + \lie{a}$ is maximal fully reducible in $\co(\mg,g)\cap \Der(\mg)$, so by \cite{Mostow} any two choices are conjugated by an element of $\CO(\mg,g)\cap\Aut(\mg)$.

It remains to prove that the eigenvalues of $N$ are rational. Let $\lambda_1,\dotsc, \lambda_p$ be its eigenvalues, with respective multiplicities $d_1,\dotsc, d_p$. Since $N$ is in $\co(\mg,g)$, we can assume $\lambda_1+\lambda_2=l=\dots = \lambda_{p-1}+\lambda_p$, $d_1=d_2,\dots, d_{p-1}=d_p$.

Consider the symmetric diagram $(\Delta,\sigma)$ associated to $\{\lambda_1,\dotsc,\lambda_p\}$ as in Lemma~\ref{lem:prenice}, and let $M_\Delta$ be the corresponding matrix. We define a matrix $N_\Delta$ by adding to $M_\Delta$ the rows
\[e_1+e_2-e_{3}-e_4, \dots ,e_1+e_2 - e_{p-1}-e_p.\]
The last step is unnecessary if $\g$ is irreducible. Indeed, the sum of the eigenspaces of $N$ that correspond to a connected component of $\Delta$ is an ideal; for $\g$ irreducible, $\Delta$ is connected, so
by Lemma~\ref{lemma:sprenice} $N_\Delta$ has the same kernel as $M_\Delta$; the following part of the proof only depends on the kernel of $N_\Delta$, so we can replace it with $M_\Delta$.

Then elements $\nu=(\nu_1,\dots, \nu_p)$  of $\ker N_\Delta$ define elements of $\co(\mg,g)\cap \Der(\mg)$ by
\[\sum \nu_t \id_{\mb_t},\]
where $\mb_t$ is the eigenspace of $N$ relative to the eigenvalue $t$. Indeed, they are in the kernel of $M_\Delta$, so they respect the grading defined by $N$, and they satisfy $\nu_1+\nu_2=k=\nu_3+\nu_4=\dots = \nu_{p-1}+\nu_p$ for some $k$, so they are in $\co(\mg,g)$.

Now fix a matrix $F$ by selecting generators of the space spanned by rows of $N_\Delta$. The fact that $\nu\in \ker F$ defines an element of $\co(\mg,g)\cap \Der(\mg)$ and the defining condition of $N$ imply that
\[\sum \nu_i d_i = \sum \nu_i d_i\overline\lambda_i,\]
i.e.
\[\sum \nu_id_i(\overline\lambda_i-1)=0, \quad \nu \in \ker F.\]

This implies that $(d_1(\overline\lambda_1-1),\dots, d_p(\overline\lambda_p-1))$ is in the span of the rows of $F$, i.e.
\[\tran F H = \begin{pmatrix}d_1(\overline\lambda_1-1)\\ \vdots \\ d_p(\overline\lambda_p-1)\end{pmatrix}\]
for some real vector $H$. Denoting by $[1]$ the vector in $\R^p$ with all entries equal to $1$, we get
\[\langle [1],\tran F H\rangle = \sum d_i(\overline\lambda_i-1)=\sum d_i(l/2-1)=n(l/2-1),\]
where $n=\dim\g$.

If we denote by $D$ the diagonal matrix with entries $d_1,\dotsc, d_p$,
we get
\[\tran F H=D(\overline\lambda-[1]), \quad \overline\lambda=\begin{pmatrix}\overline\lambda_1\\ \vdots \\ \overline\lambda_p\end{pmatrix},\] i.e.
\[\overline\lambda = D^{-1}\tran F H +[1].\]
Since $F$ has integer entries, $F\lambda=0$ implies $F\overline\lambda=0$; thus, applying $F$ to the above we get
\[0 =F D^{-1}\tran F H +F[1].\]
Now observe that $F D^{-1}\tran F$ is a real, positive definite symmetric matrix; indeed,
relative to the positive-definite real scalar product $\langle X,Y\rangle_{D^{-1}} =\tran XD^{-1}Y$ we have
\[\langle \tran FX, \tran FX\rangle_{D^{-1}}
 =
 \tran X F D^{-1}\tran F X\geq \frac1{\max \{d_1,\dotsc, d_p\}}\langle \tran FX,\tran FX\rangle,\]
so the right hand side is $\geq0$ and only vanishes when  $\tran FX=0$, i.e. $X=0$, as $F$ is surjective.

In particular, $F D^{-1}\tran F$ is invertible and
\[H=-(F D^{-1}\tran F)^{-1} F[1],\]
so $H$ has rational entries; therefore, so does $\lambda$.
\end{proof}

In analogy with the Nikolayevsky derivation, we will refer to the derivation $N$ in Theorem~\ref{thm:metricnik} as a \emph{metric Nikolayevsky derivation} of the metric Lie algebra. By construction, the eigenspaces of a metric Nikolayevsky derivation determine a grading satisfying~\eqref{eqn:co_gradation}; we can eliminate denominators and obtain:
\begin{corollary}\label{cor:metnikgrad}
Let $\g$ be a Lie algebra over $\K=\R,\C$ with a metric $g$. Then the eigenspaces of a metric Nikolayevsky derivation determine a grading $\g=\bigoplus \mb_i$; we can assume that the indices $i$ are integers. In addition, we have that $\mb_i$ is orthogonal to $\mb_j$ unless $i+j$ equals some constant $l$.
\end{corollary}
Since our goal is to identify gradings that satisfy~\eqref{eqn:co_gradation}, we observe that the grading defined by the metric Nikolayevsky derivation described in Corollary~\ref{cor:metnikgrad} is not the only possible choice. However, any other choice is ``compatible'' in the following sense:
\begin{proposition}
\label{lemma:restrictednikolayevsky}
Let $(\mg,g)$ be a metric Lie algebra endowed with a grading $\mg=\bigoplus\mb_i$ that satisfies~\eqref{eqn:co_gradation}. Then there exists a metric Nikolayevsky derivation inducing a grading
$\mg=\bigoplus\mb_i'$ with $\mb_i=\bigoplus_j \mb_i\cap \mb_j'$.
\end{proposition}
\begin{proof}
Let $N$ be a metric Nikolayevsky derivation. For every derivation $D\in\co(\mg,g)$, we have
\[\Tr(N D)=\Tr D.\]
Given a derivation $D$, write $D=D_0+D'$, where $D_0$ has degree zero, i.e $D_0|_{\mb_i}=\pi_i\circ D|_{\mb_i}$. Then $D_0$ and $D'$ are also  derivations. If $D$ is anti-self-adjoint, then so is $D_0$.

By construction, $D'$ is traceless, so we must have
\[\Tr(N D')=0,\qquad \Tr(N D_0)=\Tr D_0.\]
Now it is clear that $N_0$ satisfies these conditions, so $N_0$ is a metric Nikolayevsky derivation.

Since $N_0$ preserves each $\mb_i$, each $\mb_i$ is the direct sum of the eigenspaces $\mb_i\cap\mb_j'$.
\end{proof}

\begin{remark}
Assuming that $\g$ is irreducible, the proof of Theorem~\ref{thm:metricnik} says a little more.

The rows of $F$ take the form $-e_i-e_j+e_k$. Then $F[1]=-[1]$, so
\[n(1-l/2)=\langle [1], H \rangle= \langle [1],(FD^{-1}\tran F)^{-1}[1]\rangle;\]
in particular, since $FD^{-1}\tran{F}$ is positive definite, $l<2$. In addition, notice that
\[\Tr N = \sum d_i\lambda_i = \sum d_i \frac{l}{2}=\frac {nl}2,\]
which coincides with $\Tr N^2\geq0$. So $l\geq0$, and $l=0$ only when $N=0$.
\end{remark}

\begin{example}
\label{ex:metricnikdepends}
The metric Nikolayevsky derivation depends indeed on the metric, as shown by the following. Consider the Lie algebra
$(0,0,12,13)$ and the two metrics
\[g_1=e^1\odot e^4+e^2\odot e^3,\quad g_2= e^1\odot e^3+e^2\odot e^4.\]
The generic derivation has the form
\[\begin{pmatrix}
-\lambda_3+\lambda_7&0&0&0\\
\lambda_1&2 \lambda_3-\lambda_7&0&0\\
\lambda_2&\lambda_6&\lambda_3&0\\
\lambda_4&\lambda_5&\lambda_6&\lambda_7
\end{pmatrix}.\]
The metric Nikolayevsky derivation relative to $g_1$ is $\diag(\frac13,\frac23,1,\frac43)$, which coincides with the Nikolayevsky derivation; the matrix $F$ appearing in the proof of Theorem~\ref{thm:metricnik} takes the form
\[\begin{pmatrix}
-2 & 1 & 0 &0\\
-1& -1 &1 & 0 \\
-1 & 0 & -1 & 1 \\
0 & -2 & 0 & 1
  \end{pmatrix}.\]
The metric Nikolayevsky derivation relative to $g_2$ is $\diag(\frac23,0,\frac23,\frac43)$.

The fact that the two metric Nikolayevsky derivations have different eigenvalues shows that $g_1$ and $g_2$ are not related by an isomorphism of the Lie algebra.

Notice that the metrics $g_1$ and $g_2$ are not ad-invariant.
\end{example}

For completeness, we prove the following:
\begin{proposition}
Let $(\g,g)$ be a metric Lie algebra over $\K=\R,\C$ with metric Nikolayevsky derivation $N$. Then:
\begin{enumerate}[label=(\roman*)]
\item If $(\hat \g,\hat g)$ is another metric Lie algebra over $\K$ with metric Nikolayevsky $\hat N$, then $N+\hat N$  is a metric Nikolayevsky derivation on $(\g\oplus\hat\g,g+\hat g)$.
\item If $\K=\R$, on the complexification $(\g^\C,\tilde g)$, $N^\C$ is a metric Nikolayevsky derivation.
\item If $\K=\C$, the metric Nikolayevsky derivation on  $(\overline\g,\overline g)$ is $\overline N$.
\item If $\K=\C$, the metric Nikolayevsky derivation on  $(\g^\R,\Re g)$ is $N^\R$.
\end{enumerate}
\end{proposition}
\begin{proof}
The first item follows from the fact that the projection
\[\Der (\g\oplus\hat\g)\cap \co(\g\oplus\hat\g,g+\hat g)\to(\Der(\g)\cap\co(\g,g)\oplus(\Der(\hat\g)\cap\co(\hat\g, \hat g))\]
is surjective and its kernel consists of traceless derivations.

The second follows from  the identifications
\[\Der(\g^\C)=\Der(\g)\otimes\C, \quad \co(\g^\C,\tilde g)=\co(\g,g)\otimes\C\]
(see the second item in  Lemma~\ref{lemma:elementary}). Indeed, these identifications imply that an endomorphism of $\g$ is a metric Nikolayevsky derivation if and only if its complexification is a metric Nikolayevsky derivation of $\g^\C$.

Similarly, the third follows from $\Der(\overline \g)=\overline{\Der (\g)}$, $\co(\overline \g,\overline g)=\overline{\co(\g,g)}$.

For the last one, we identify the complexification $((\g^\R)^\C,\widetilde{\Re g})$ with $(\g,g)\oplus(\overline\g,\overline g)$. By the first part of the proof, $N+\overline N$ is the metric Nikolayevsky derivation of $(\g^\R)^\C$; since $N+\overline N$ is the complexification of $N^\R$, it follows that $N^\R$ is the metric Nikolayevsky derivation of $(\g^\R,\Re g)$.
\end{proof}
\begin{remark}
The same arguments also prove the analogous statement for the ordinary Nikolayevsky derivation.
\end{remark}

\subsection{The metric Nikolayevsky derivation in the ad-invariant case}
In the ad-invariant case, we can use the metric Nikolayevsky derivation introduced in Theorem~\ref{thm:metricnik} to obtain a sufficient condition for the metric to be solitary.

In contrast with Example~\ref{ex:metricnikdepends}, there is a  class of metric Lie algebras for which the metric Nikolayevsky derivation does not depend on the metric, as it coincides with the Nikolayevsky derivation.

Let $e_1,\dotsc, e_n$ be a basis of a Lie algebra $\mg$. A metric is said to be \emph{$\sigma$-diagonal} if it has the form $\sum_i g_ie_i\odot e_{\sigma_i}$ for some $g_i\in \K^*$, with $\sigma$ an involution of $\{1,\dotsc, n\}$ (see \cite{ContiRossi:RicciFlat,ContiDelBarcoRossi}).
\begin{proposition}
Let $\{e_1,\dotsc, e_n\}$ be a basis on an irreducible Lie algebra $\mg$ that diagonalizes the Nikolayevsky derivation $N$. Assume that there exists a $\sigma$-diagonal ad-invariant metric $g$. Then the metric Nikolayevsky derivation coincides with the Nikolayevsky derivation.
\end{proposition}
\begin{proof}
It suffices to show that $N$ lies in $\co(\mg,g)$. Let $\lambda_i,\dots,\lambda_n$ be the eigenvalues of $N$ with respect to the basis $\{e_1,\dotsc, e_n\}$; then
\[N+N^*=\sum_i\lambda_i (e^i\otimes e_i+ e^{\sigma_i}\otimes e_{\sigma_i})=\sum (\lambda_i + \lambda_{\sigma_i})e^i\otimes e_i;\]
therefore, it suffices to show that each $\lambda_i+\lambda_{\sigma_i}$ equals some constant $l$.

Let $\Delta$ be the digraph with nodes $\{1,\dotsc, n\}$ and arrows $i\xrightarrow{j}k$ when $c_{ijk}=g([e_i,e_j],e_{\sigma_k})$ is nonzero, i.e. when $[e_i,e_j]$ has a component along $e_k$. The ad-invariant condition gives
\[c_{ijk}=g([e_i,e_j],e_{\sigma_k})=g([e_{\sigma_k},e_i],e_{j})=c_{\sigma_ki\sigma_j},\]
so if $i\xrightarrow{j}k$ is an arrow, then $i\xrightarrow{\sigma_k}\sigma_j$ is also an arrow. Thus, $(\Delta,\sigma)$ is a symmetric  diagram. In addition $\Delta$ is connected: indeed, if $i$ and $k$ are in different connected components of $\Delta$, then $c_{ijk}=0$, so every connected component defines an ideal.

Since $N$ is a derivation, we have
\[g(N[e_i,e_j]-[Ne_i,e_j]-[e_i,Ne_j],e_{\sigma_k})=(\lambda_k-\lambda_i-\lambda_j )c_{ijk}=0.\]
This shows that $(\lambda_1,\dotsc, \lambda_n)$ is in $\ker M_\Delta$. By Lemma~\ref{lemma:sprenice}, there is a constant $l$ such that $\lambda_i+\lambda_{\sigma_i}=l$ for all $i$, which is what we had to prove.
\end{proof}
Recall that a Lie algebra is called \emph{nice} if it has a basis $\{e_1,\dotsc, e_n\}$ such that each Lie bracket $[e_i,e_j]$ is a multiple of an element of the basis and, denoting by $\{e^1,\dotsc, e^n\}$ the dual basis and by $d$ the Chevalley-Eilenberg differential, each $e_i\hook de^j$ is a multiple of some element of the dual basis (see \cite{Nikolayevsky,LauretWill:EinsteinSolvmanifolds}). We then say that $\{e_1,\dotsc, e_n\}$ is a \emph{nice basis}. A graded Lie algebra $\g=\bigoplus\mb_i$ such that each $\mb_i$ has dimension one is always nice, but the converse is not true.

On a nice Lie algebra, the Nikolayevsky derivation can be assumed to be diagonal relative to a nice basis (see \cite{Nikolayevsky}); as a corollary, we have:
\begin{corollary}\label{cor:metricNikIsNik}
On an irreducible Lie algebra with an ad-invariant metric which is  $\sigma$-diagonal with respect to a nice basis, the metric Nikolayevsky derivation coincides with the Nikolayevsky derivation.
\end{corollary}

\begin{example}
\label{ex:adinvmetricnikdepends}
Even if we only consider ad-invariant metrics, the metric Nikolayevsky derivation  depends on the choice of the metric. Consider the cotangent of $\su(2)$, endowed with the ad-invariant metrics $g,h$ of Example~\ref{example:Tstarsu2}. The generic derivation has the form
\[D=\begin{pmatrix}
0&- \lambda_2&- \lambda_5&0&0&0\\
\lambda_2&0&- \lambda_6&0&0&0\\
\lambda_5&\lambda_6&0&0&0&0\\
0&- \lambda_1&- \lambda_3&\lambda_7&- \lambda_2&- \lambda_5\\
\lambda_1&0&- \lambda_4&\lambda_2&\lambda_7&- \lambda_6\\
\lambda_3&\lambda_4&0&\lambda_5&\lambda_6&\lambda_7
\end{pmatrix}.\]
Whilst all derivations are in $\co(\g,g)$, we have
\[\Der(\g)\cap \co(\g,h)=\{D\in\Der\g\st \lambda_7=0\}.\]
It follows that the metric Nikolayevsky derivations of $g$, $h$ are respectively
\[N_{g}=\diag(0,0,0,1,1,1), \qquad N_{h}=0.\]
\end{example}

We now establish a condition for a Lie algebra to be solitary in terms of the metric Nikolayevsky derivation.
\begin{theorem}
\label{thm:uniquenesstheorem}
Let $\g$ be a Lie algebra over $\K=\R,\C$ with an ad-invariant metric such that the eigenvalues of the metric Nikolayevsky derivation are positive. Then the metric  is \rigid.
\end{theorem}
\begin{proof}
The eigenspaces of the metric Nikolayevsky derivation determine a grading $\g=\bigoplus \mb_i$ satisfying~\eqref{eqn:co_gradation}, as proved in Corollary~\ref{cor:metnikgrad}; we can assume that the indices $i$ are positive integers. In particular, $\mb_i$ is orthogonal to $\mb_j$ unless $i+j$ equals some constant $l$. By construction, the indices range between $1$ and $l-1$, and thus every degree $l$ map is trivial. Hence,  Lemma~\ref{lemma:uniquenesslemma} applies.
\end{proof}

\begin{example}
The Boidol Lie algebra
\[(0,-12,13,-23)\]
is solvable and nice; it admits the $\sigma$-diagonal ad-invariant metric $e^1\odot e^4+e^2\odot e^3$. The metric Nikolayevsky derivation is
$\diag(0,\frac23,\frac23,\frac43)$, so Theorem~\ref{thm:uniquenesstheorem} does not apply. Nevertheless, this Lie algebra is \rigid, since it admits a unique ad-invariant metric \cite{dBOV} (see also Lemma~\ref{lm:solv4dim} below).
\end{example}

Nilpotent examples of \rigid Lie algebras such that the metric Nikolayevsky derivation does not have positive eigenvalues will be constructed in Section~\ref{sec:Cotangents}.

It follows from the above example that the the converse of Theorem~\ref{thm:uniquenesstheorem} is not true: given a solitary ad-invariant metric $g$, the metric Nikolayevsky derivation is not necessarily positive. Notice however that it cannot be zero, because the \rigid condition implies the existence of a derivation $D$ such that $\id=D+D^*$, giving a derivation $D\in\co(\g,g)$ with nonzero trace.

\section{Cotangent Lie algebras}\label{sec:Cotangents}
In this section we study the solitary condition for cotangent Lie algebras, which one can always endow with a canonical ad-invariant metric. In the first part of the section, we characterize cotangent Lie algebras for which this (and thus every) ad-invariant metric is solitary. The second part is devoted to the computation of the metric Nikolayevsky derivation corresponding to the canonical ad-invariant metric, which coincides with the Nikolayevsky derivation of the cotangent Lie algebra.

\subsection{The solitary condition for the canonical metric}
Recall that if $\g$ is a Lie algebra, then the cotangent Lie algebra $T^*\g$ is defined as a semidirect product $\g\ltimes\g^*$, where $\mg^*$ is an abelian ideal, and the action of $\mg$ on $\mg^*$ is given by
\[[X,\alpha]=\ad_X^*\alpha=-\alpha\circ\ad_X, \quad\text{ for all }X\in \mg,\; \alpha\in \mg^*.\]
The cotangent Lie algebra has a canonical ad-invariant metric $g$ induced by the pairing of $\g$ with $\g^*$, namely
\begin{equation}\label{eqn:CanonicalCotangentMetric}
g(X+\alpha,Y+\beta)=\alpha(X)+\beta(Y),\quad\text{ for all }X,Y\in \mg,\; \alpha,\beta\in \mg^*.
\end{equation}
Let $P\colon T^*\mg\to T^*\mg$ be the projection onto $\mg^*$. If $P^*$ denotes its adjoint with respect to $g$, then $P+P^*=\id$. Therefore, $g$ is weakly solitary if and only if it is solitary.
In addition, on an irreducible real Lie algebra the \rigid condition implies uniqueness of the ad-invariant metric up to sign (see Theorem~\ref{thm:rigidimpliesuniqueness}). However, on a cotangent $\g\ltimes \g^*$ the reflection across $\g$, i.e. $(x,\eta)\mapsto (x,-\eta)$, is a Lie algebra isomorphism, which relates the canonical ad-invariant metric to its opposite. Therefore, if $T^*\mg$ is irreducible and solitary, then it admits a unique ad-invariant metric up to isomorphisms.

Every cotangent $T^*\mg= \g\ltimes\g^*$ admits the grading
\begin{equation}
 \label{eqn:grading_of_cotangent}
 \g\ltimes\g^*=\mb_0\oplus\mb_1, \text{ where }\quad \mb_0=\g,\ \mb_1=\g^*,
\end{equation}
which satisfies $\mb_i\bot\mb_j$ if $i+j\neq 1$ with respect to $g$, so~\eqref{eqn:co_gradation} holds for $l=1$. Notice that $\mb_i$ is the eigenspace corresponding to the eigenvalue $i$ of the projection $P$, for $i=0,1$.

The following notation will be useful. Given a linear map  $\phi\colon\g\to\g^*$, its transpose $\tran\phi\colon\g^{**}\to\g^*$ can be identified with a map from $\g$ to $\g^*$. We say $\phi$ is \emph{symmetric} if $\phi=\tran\phi$ in this sense, i.e. $\phi(X)(Y)=\phi(Y)(X)$ for all $X,Y\in\mg$. We say that $\phi$ is \emph{ad-invariant} (resp. a \emph{derivation}) if so is the endomorphism of the cotangent $\g\ltimes\g^*$ obtained from $\phi$ by extending by zero.

In the above notation, a linear map $\phi\colon\g\to\g^*$ is symmetric if and only if its extension to $T^*\mg$, defining it by zero on $\mg^*$, is $g$-self-adjoint.
The ad-invariant condition for a linear map $\phi\colon\g\to\g^*$ is equivalent to
\[\phi[X,Y]=[X,\phi(Y)],\qquad \text{ for all }X,Y\in\g,\]
i.e. $\phi\circ \ad_X = \ad_X^*\circ\phi$ for all $X\in\g$.
Similarly, a linear map $D\colon\g\to\g^*$ is a derivation if and only if
\begin{align*}
D[X,Y](Z)= g( D[X,Y],Z) &=
g( [DX,Y],Z) + g( [X,DY],Z)\\
&=DX[Y,Z] -DY[X,Z], \qquad \text{ for all }X,Y,Z\in \mg,
\end{align*}
i.e.
\[D\circ\ad_X -\ad_X^*\circ D +d(DX)=0,\]
where $d(DX)$ is the linear map $\g\to\g^*$ determined by the Chevalley-Eilenberg differential of $DX\in\g^*$, namely  $g(d(DX)Y,Z)= -DX([Y,Z])$, for all $Y,Z\in\mg$.

\begin{proposition}
\label{prop:Tstarsolitary}
The canonical metric on $T^*\g$ is \rigid if and only if every ad-invariant symmetric map $\phi\colon \g\to\g^*$ has the form $D+\tran D$, with $D\colon\g\to\g^*$ a derivation.
\end{proposition}
\begin{proof}
Consider the grading~\eqref{eqn:grading_of_cotangent} of $T^*\mg$, which satisfies~\eqref{eqn:co_gradation} for the canonical metric and for $l=1$.

Assume that the canonical metric $g$ on $T^*\mg$ is \rigid and let  $\phi_1\colon\mb_0\to\mb_1$ be an ad-invariant symmetric map. Let $\phi\colon T^*\g\to T^*\g$ be the extension by zero on $\mb_1$ of $\phi_1$; by definition, $\phi$ is ad-invariant on $T^*\mg$ and self-adjoint with respect to $g$. By hypothesis, $\phi$ has the form $D+D^*$ for some derivation $D\colon\mg\to\mg^*$ of the cotangent. However, with respect to the splitting~\eqref{eqn:grading_of_cotangent}, we have that the degree one component of $D+D^*$ (resp. $\phi$) is $D_1+D_ 1^*$ (resp. $\phi_1$), where $D_1$ is the degree one component of $D$ which is also a derivation by Proposition~\ref{prop:thegradedpart}. Finally, we observe that $\tran{D_1}=D_1^*$, hence we have $\phi_1= D_1+\tran D_1$.

For the converse, we apply Lemma~\ref{lemma:uniquenesslemma} to the grading~\eqref{eqn:grading_of_cotangent}. Any ad-invariant $g$-self-adjoint map $\phi\colon T^*\mg\to T^*\mg$ of degree $l=1$ is zero on $\mg^*$ and has image contained $\mg^*$. Hence $\phi\colon\mg\to\mg^*$ and it is ad-invariant and symmetric in the notations above. Hence $\phi=D+\tran{D}$ for some derivation $D\colon\mg\to\mg^*$. Since $\tran{D}=D^*$ with respect to $g$, Lemma~\ref{lemma:uniquenesslemma} applies and the canonical metric $g$ on $T^*\mg$ is \rigid.
\end{proof}

It will be useful to have a name for Lie algebras satisfying the condition of Proposition~\ref{prop:Tstarsolitary}, so we give the following:
\begin{definition}\label{def:adsolitary}
A Lie algebra $\g$ is \emph{\adsolitary} if for every ad-invariant symmetric $\phi\colon\g\to\g^*$ there is a derivation $D\colon\g\to\g^*$ such that $\phi=D+\tran D$.
\end{definition}

Using Definition~\ref{def:adsolitary}, Proposition~\ref{prop:rigid_independent} can be rephrased  as follows:
\begin{corollary}
\label{cor:cotangentunique}
A Lie algebra $\g$ is \adsolitary if and only if $T^*\g$ is \rigid.
\end{corollary}

\begin{remark}
\label{rk:WWadsolitary}
Any $\phi\colon \g\to\g^*$ such that $\phi(\g')=0$ and $\im\phi\subset\Ann\g'$ is both ad-invariant and a derivation, because at the level of $\g\ltimes\g^*$ it satisfies~\eqref{eqn:phiWZ}.
\end{remark}

In analogy to Proposition~\ref{prop:decomposablesolitary}, we have:
\begin{proposition}
\label{prop:reducibleadsolitary}
If $\g$ is reducible, then $\g$ is \adsolitary if and only if each irreducible component is \adsolitary.
\end{proposition}
\begin{proof}
Let $\g=\g_{1}\oplus\g_{-1}$, and let $\phi\colon\g\to\g^*$ be ad-invariant and symmetric; decompose $\phi$ as $\phi_++\phi_-$, where
\[\phi_\pm(\g_\epsilon)\subset \g_{\pm \epsilon}, \quad \epsilon\in\{1,-1\}.\]
It is clear that $\phi_+$ is again ad-invariant and symmetric. Therefore, so is $\phi_-$.

For $X,Y\in\g_\epsilon$ we have
\[\phi_-([X,Y])=[X,\phi_-(Y)]\in [\g_\epsilon,\g_{-\epsilon}]=0;\]
it follows that $\phi_-(\g')=0$; therefore, $\phi_-$ is itself a symmetric derivation (see also Remark~\ref{rk:WWadsolitary}).

It follows that $\phi$ can be written in the form $D+D^*$ if and only if so can $\phi_+$. Since $\phi_+$ preserves $\g_1$ and $\g_{-1}$, we see that $\g$ is \adsolitary if and only if $\g_1$ and $\g_{-1}$ are \adsolitary.
\end{proof}

The following lemma shows that when $\mg$ admits an ad-invariant metric, ad-invariant symmetric maps from $\mg$ to $\mg^*$ can be identified with ad-invariant self-adjoint endomorphisms of $\mg$.

\begin{lemma}
\label{lemma:cotangentofadinvariant}
Let $\g$ be a Lie algebra with an ad-invariant metric $g$, let $f\colon \g\to\g^*$ be a linear map, and let $f^\sharp\colon\g\to\g$ be the map satisfying
\[ f(X)(Y)=g(f^\sharp(X),Y) \quad \text{ for all }
X,Y\in\mg.\]
Then $f$ is ad-invariant (resp. a derivation) if and only if $f^\sharp$ is ad-invariant (a derivation); in addition, $(\tran f)^\sharp=(f^\sharp)^*$.
\end{lemma}
\begin{proof}
For $X,Y,Z$ in $\g$, we have
\begin{multline*}
g(f^\sharp[X,Y]-[f^\sharp X,Y],Z)=
 g(f^\sharp[X,Y],Z)-g(f^\sharp X,[Y,Z])\\
 =(f[X,Y])(Z)
 -(f X)([Y,Z])
= (f[X,Y]-[f X,Y])(Z),
\end{multline*}
so nondegeneracy of $g$ implies that $f$ is ad-invariant if and only if so is $f^\sharp$. A similar reasoning proves that $f$ is a derivation if and only if $f^\sharp$ is a derivation.

The last statement follows from the following equality, for $X,Y\in\mg$:
\[g((f^\sharp)^* X,Y)=g(X,f^\sharp Y)=( f Y)(X) = (\tran f X)(Y)=g((\tran f)^\sharp X,Y).\qedhere\]
\end{proof}
\begin{remark}\label{rem:notadinv}
Given a Lie algebra $\mg$, one can easily prove that there is a bijective ad-invariant symmetric $\phi\colon\mg\to\mg^*$ if and only if $\mg$ admits an ad-invariant metric. Indeed, given such a $\phi$, then $g(X,Y)=\phi(X)(Y)$ defines an ad-invariant metric on $\mg$. The converse is proved by taking an ad-invariant metric $g$ on $\mg$ and $\phi\colon\g\to\g^*$ such that $\phi^\sharp=\id$ as in Lemma~\ref{lemma:cotangentofadinvariant}.
\end{remark}

A Lie algebra $\mg$ admitting ad-invariant metrics can be ``\rigid'' in two different ways: $\mg$ can be \rigid or it can be $T^*$-\rigid, but the distinction turns out to be apparent:
\begin{theorem}\label{thm:rigiffTrig}
Let $\g$ be a Lie algebra with an ad-invariant metric $g$. Then $\g$ is \rigid if and only if it is \adsolitary.
\end{theorem}
\begin{proof}
Assume that $g$ is \rigid. Let $\phi\colon\g\to\g^*$ be ad-invariant and symmetric, and write
$\phi^\sharp\colon\g\to\g$ as in Lemma~\ref{lemma:cotangentofadinvariant}; then $\phi^\sharp$ is ad-invariant and self-adjoint. Therefore, we can write
\[\phi^\sharp=D+D^*, \quad D\in\Der(\g),\]
where $D^*$ is the adjoint relative to the metric $g$.
Defining
\[D^\flat\colon\g\to\g^*, \quad  g(D(X),Y)=(D^\flat X)(Y) ,\quad X,Y\in\mg,\]
and applying Lemma~\ref{lemma:cotangentofadinvariant} again, we see that $D^\flat\colon\g\to\g^*$ is a derivation and $(D^\flat)^*=(D^*)^\flat$; consequently,
\[\phi=(\phi^\sharp)^\flat = D^\flat + (D^*)^\flat = D^\flat + (D^\flat)^*.\]

Conversely, suppose that $\g$ is \adsolitary and let
$\phi\colon\g\to\g$ be self-adjoint and ad-invariant. Then $\phi^\flat\colon\g\to\g^*$ is ad-invariant and symmetric; therefore,
\[\phi^\flat=D+D^*,\]
where $D\colon\g\to\g^*$ is a derivation, and $\phi=D^\sharp + (D^\sharp)^*$, so $g$ is \rigid.
\end{proof}

As an immediate consequence, we obtain:
\begin{corollary}\label{cor:simplenotTsol}
A semisimple Lie algebra is never $T^*$-\rigid.
\end{corollary}
\begin{proof}
On a semisimple Lie algebra, the Killing form defines an ad-invariant metric which is not \rigid, since derivations are inner and hence traceless (see Remark~\ref{remark:tracess_not_rigid}). So such a Lie algebras is not $T^*$-solitary by Theorem~\ref{thm:rigiffTrig}.
\end{proof}

\begin{remark}
Baum and Kath proved in \cite{BaumKath} that the cotangent of a semisimple Lie algebra possesses ad-invariant metrics which are not related by an automorphism, by using the Malcev–Harish-Chandra theorem.
\end{remark}

The rest of the section will focus on conditions for a Lie algebra $\mg$ (not necessarily admitting ad-invariant metrics) to be $T^*$-\rigid.
We start by introducing a technical lemma.

It is easy to verify that the kernel of every ad-invariant map $\phi\colon\mg\to\mg^*$ is an ideal of $\mg$, so one can consider the quotient Lie algebra $\mg/\ker \phi$.

\begin{lemma}
\label{lemma:liftderivation}
Let $\phi\colon \g\to\g^*$ be ad-invariant and symmetric. If $\g/\ker\phi$ is \adsolitary, then $\phi=D+\tran D$ for some derivation $D\colon\g\to\g^*$.
\end{lemma}
\begin{proof}
If $\phi$ is injective, the statement is trivial. Otherwise, let $U=\ker\phi$ and note that $\phi(\mg)\subset\Ann U$, where $\Ann U$ denotes the annihilator of $U$ in $\mg^*$. Define $\psi\colon \g/U\to (\g/U)^*\cong \Ann U$ by the commutativity of the diagram
\[\xymatrix{ \g\ar[r]^\phi\ar[d] & \Ann U\ar[d]^\cong \\ \g/U\ar[r]^\psi & (\g/U)^*}\]
Since $\phi$ is ad-invariant and symmetric, so is $\psi$. Therefore $\psi=\eth+\tran\eth$ with $\eth\colon\g/U\to(\g/U)^*$ a derivation.

Now lift $\eth$ to $D\colon\g\to\g^*$ by composing with projection and inclusion. We have
\[D[X,Y]=\eth([X,Y]+U)=[\eth(X+U),Y]+[X,\eth(Y+U)]=[DX,Y]+[X,DY].\]
In addition
\[\tran D(X)(Y)=D(Y)(X)=\eth(Y+U)(X)=\tran\eth(X+U)(Y).\]
Thus, $\phi=D+\tran D$.
\end{proof}

\begin{proposition}
\label{prop:quotientweak}
Let $\g$ be a Lie algebra over $\K=\R,\C$ which is not $T^*$-solitary. Then there exists an ideal $U\subset\g$ such that $\g/U$ has a nonsolitary ad-invariant metric, and in addition:
\begin{enumerate}[label=(\roman*)]
 \item if $\K=\C$, the metric on $\g/U$ is weakly solitary;
 \item if $\K=\R$, either the metric on $\g/U$ is weakly solitary, or $\g/U$ is the underlying real Lie algebra of a complex Lie algebra with a weakly solitary ad-invariant metric.
\end{enumerate}
\end{proposition}
\begin{proof}
If $\g$ has dimension one, there is nothing to prove.

For $\g$ of dimension $n>1$, we proceed by induction. Let $\phi\colon\g\to\g^*$ be a symmetric ad-invariant map which cannot be written as $\phi=D+\tran D$.

If $\g$ does not have ad-invariant metrics, $\phi$ has a nontrivial kernel $U$ by Remark~\ref{rem:notadinv}. The quotient $\g/U$ is not $T^*$-solitary by Lemma~\ref{lemma:liftderivation}, so by induction $\g/U$ has a quotient admitting a weakly solitary, nonsolitary ad-invariant metric; since quotients of $\g/U$ are also quotients of $\g$, the statement is proved.

If $\g$ has an ad-invariant metric $g$, this metric cannot be solitary by Theorem~\ref{thm:rigiffTrig}. If $\g$ is reducible, decompose $\g$ into the orthogonal sum of irreducible Lie algebras as $\g=\g_1\oplus\dots \oplus\g_k$; we have that one of the $\g_i$ is nonsolitary, hence not $T^*$-solitary; by the inductive hypothesis, it has a quotient as in the statement, which is also a quotient of $\g$.

Assume now that $\g$ is irreducible. If $g$ is weakly solitary, there is nothing to prove; otherwise, Corollary~\ref{cor:phihasoneortwoeigenvalues} implies that either
\begin{enumerate}
\item there is a nilpotent, ad-invariant self-adjoint $\psi\colon\g\to\g$ which does not take the form $\psi=D+D^*$; or
\item $\K=\R$ and $\g$ has an ad-invariant, self-adjoint complex structure $J$.
\end{enumerate}
In the first case, by Lemma~\ref{lemma:cotangentofadinvariant} we have that $\psi^\flat\colon\g\to\g^*$ is ad-invariant and does not take the form $\psi^\flat=D+\tran D$ for a derivation $D\colon\g\to\g^*$. By Lemma~\ref{lemma:liftderivation}, $\g/\ker \psi$ is not $T^*$-solitary, so by inductive hypothesis it has a quotient as in the statement.

In the second case, $\g$ is the real Lie algebra underlying a complex Lie algebra $\mh$ with an ad-invariant metric. By Proposition~\ref{prop:solitaryiffrealsolitary}, $\mh$ is not solitary, so by the inductive hypothesis it has a quotient $\mh/U$ which is weakly solitary but not solitary. Then $U$ is also an ideal in $\g$, and $\g/U$ is nonsolitary by Proposition~\ref{prop:solitaryiffrealsolitary}.
\end{proof}

Another useful consequence of Lemma~\ref{lemma:liftderivation} is the following:
\begin{proposition}
\label{prop:impliesadsolitary}
Let $\g=\bigoplus\mb_i$ be a Lie algebra graded over $\Z$. If either:
\begin{enumerate}
 \item the indices $i$ are positive; or
 \item the indices $i$ are nonnegative and $\mb_0$ is \adsolitary,
\end{enumerate}
then $\g$ is \adsolitary.
\end{proposition}
\begin{proof}
Let $l$ be an integer which is greater than $2i$ for every index $i$ of the grading.

The splitting $\g=\bigoplus\mb_i$ induces a splitting $\g^*=\bigoplus\mb_i^*$, and \[\g\ltimes \g^*=\bigoplus \mb_i \oplus \bigoplus\mb_i^*.\] We can make this direct sum into a gradation by assigning degree $l-i$ to $\mb_i^*$. This is a grading because when $X$ is in $\mb_i$, then $\ad_X(\mb_j)\subset \mb_{i+j}$, so $\ad^*_X(\mb_k^*)\subset \mb_{k-i}^*$, and $\g^*$ is abelian. Also note that the canonical ad-invariant metric~\eqref{eqn:CanonicalCotangentMetric} trivially satisfies~\eqref{eqn:co_gradation}.

By construction, if $i$ is an index of $\g=\bigoplus\mb_i$, then $l+i$ is only an index if $i=0$, and $(l-i)+l>l$ is not an index. Thus, the degree $l$ component of $\phi$ is an ad-invariant symmetric map $\phi_l\colon\mb_0\to\mb_0^*$, and since $\mb_0$ is \adsolitary we can apply Lemma~\ref{lemma:uniquenesslemma} to have that $T^*\g$ is \rigid. Finally, we apply Corollary~\ref{cor:cotangentunique} to conclude.
\end{proof}

\begin{corollary}
\label{cor:2stepimpliesadsolitary}
Every $2$-step nilpotent Lie algebra is $T^*$-solitary.
\end{corollary}
\begin{proof}
Since the commutator is contained in the center, we have a positive grading $\mb_1\oplus\mb_2$, with $\mb_2$ equal to the commutator. Proposition~\ref{prop:impliesadsolitary} implies the statement.
\end{proof}

Recall that a Riemannian nilsoliton is a nilpotent Lie algebra with a Riemannian metric satisfying $\ric=\lambda \id+D$ for some derivation $D$. It is well known that the derivation $D$ is then a multiple of the Nikolayevsky derivation; in fact, a nilpotent Lie algebra is a nilsoliton if and only the Nikolayevsky derivation has positive eigenvalues. By considering the grading of $\mg$ induced by this derivation, Proposition~\ref{prop:impliesadsolitary} immediately implies:
\begin{corollary}
\label{cor:nilsolitonsareadsolitary}
Riemannian nilsolitons are $T^*$-solitary.
\end{corollary}
It is worth noting that all nilpotent Lie algebras of dimension $\leq 6$  are Riemannian nilsolitons (see \cite{Will:RankOne}). However, Proposition~\ref{prop:impliesadsolitary} applies more generally:
\begin{example}
Consider the $11$-dimensional nilpotent Lie algebra
\[(0,0,0,0,0,0,0,e^{12},e^{13}+e^{24}+e^{57}+e^{68},e^{19}+e^{34}+e^{56}+e^{78})\]
This is a nice Lie algebra whose Nikolayevsky derivation is
\[\diag\left(0,\frac76,\frac76,0,\frac7{12},\frac7{12},\frac7{12},\frac7{12},\frac76,\frac76,\frac76\right).\]
Therefore, it admits a nonnegative grading with $\mb_0$ abelian, and is \adsolitary by Proposition~\ref{prop:impliesadsolitary}.
\end{example}
The converse of Proposition~\ref{prop:impliesadsolitary} does not hold, as shown by the following examples.
\begin{example}
\label{ex:123457E}
Consider  the nilpotent Lie algebra
\[(0,0,e^{12},e^{13},e^{14},e^{23}+e^{15},e^{23}+e^{24}+e^{16}).\]
This $7$-dimensional Lie algebra is characteristically nilpotent. Any such Lie algebra does not admit nontrivial gradings; indeed, a grading $\g=\bigoplus_t\mb_t$ induces a derivation $\sum t\id_{\mb_t}$, which is only nilpotent if the grading takes the form $\g=\mb_0$.

Thus, Proposition~\ref{prop:impliesadsolitary} does not apply to this case. However, we
can check that this Lie algebra is \adsolitary by by observing that ad-invariant symmetric maps $\g\to\g^*$ are zero on $\g'$, so Remark~\ref{rk:WWadsolitary} applies.
A similar argument applies to the $7$-dimensional characteristically nilpotent Lie algebra
 $(0,0,e^{12},e^{13},0,e^{25}+e^{14}+e^{23},-e^{34}+e^{26}+e^{15})$.
\end{example}

\begin{example}\label{ex:negativeeigenvalue}
The $9$-dimensional nice nilpotent Lie algebra
\[(0,0,0,e^{12},e^{13},e^{23},e^{24},e^{35}+e^{26},e^{27}+e^{15})\]
has Nikolayevsky derivation
\[\diag\left(
\frac{1}{7},0,-\frac{1}{7},\frac{1}{7},0,-\frac{1}{7},\frac{1}{7},-\frac{1}{7},\frac{1}{7}\right);\] moreover, every nonzero semisimple derivation has both positive and negative eigenvalues, so there is no nontrivial gradation with nonnegative indices.

We can prove that this Lie algebra is $T^*$-solitary by observing that all ad-invariant symmetric maps $\g\to\g^*$ vanish on $\Span{e_7,e_8,e_9}$, and applying Lemma~\ref{lemma:liftderivation}. Indeed, the quotient $\g/\Span{e_7,e_8,e_9}$ is two-step nilpotent, hence $T^*$-solitary by Corollary~\ref{cor:2stepimpliesadsolitary}.
\end{example}

\subsection{The Nikolayevsky derivation on $T^*\g$}
For completeness, we give an explicit formula for the Nikolayevsky derivation of $T^*\g$ in terms of the Nikolayevsky derivation of $\g$.

Let us first describe the structure of $\Der(\mg\ltimes\mg^*)$.
Let $D\in \gl(\mg\ltimes\mg^*)$ and consider the induced linear maps
\begin{equation}\label{eq:derivcot}
D=
\begin{pmatrix}
D_1&D_3\\
D_2&D_4
\end{pmatrix}.
\end{equation}
\begin{lemma}\label{lm:D1234}
$D$ is a derivation of $\mg\ltimes\mg^*$ if and only if for every $X,Y\in \mg$ and $\alpha, \beta\in \g^*$ the following conditions hold:
\begin{enumerate}[label=(\roman*)]
\item $D_1\in \Der(\mg)$ and $[D_4,\ad_X^*]=\ad_{D_1X}^*$;
\item $D_2[X,Y]=\ad_X^*(D_2Y) -\ad_Y^*(D_2X)$;
\item $D_3\circ\ad^*_X=\ad_ X\circ D_3$ and $\ad^*_{D_3\alpha}(\beta)=\ad^*_{D_3\beta}(\alpha)$.
\end{enumerate}
In particular, derivations of degree zero with respect to~\eqref{eqn:grading_of_cotangent} have the form
\[
\begin{pmatrix}
D_1&0\\
0&-\tran D_1+\tran\phi
\end{pmatrix}.\]
where $D_1\colon\g\to\g$ is a derivation and $\phi\colon\g\to\g$ is ad-invariant.
\end{lemma}
\begin{proof}
Suppose that $D\in \Der(T^*\mg)$, that is
\begin{equation}
\label{eq:duv}
D[Z,W]=[DZ,W]+[Z,DW],\quad \text{ for all }Z,W\in T^*\mg.
\end{equation}

Taking $Z=X, W=Y\in \mg$ in~\eqref{eq:duv}, and projecting the equality onto $\mg$ and $\mg^*$, we get that $D_1\in \Der(\mg)$ and
$D_2[X,Y]=(D_2Y)\circ \ad_X-(D_2X)\circ \ad_Y$.
Moreover, putting $Z=X\in \mg$ and $W=\alpha\in \mg^*$ in the same equation we get $D_3\circ \ad_X^*=\ad_X\circ D_3$ and $[D_4,\ad_X^*]=\ad_{D_1X}^*$. Finally, taking $Z=\alpha,W=\beta\in \mg^*$ we obtain
$\beta\circ \ad_{D_3\alpha}=\alpha\circ \ad_{D_3\beta}$.

Now let $D$ be a degree zero derivation, that is, with $D_2=D_3=0$. Then, $D_1$ is a derivation and thus $[\tran D_1,\ad_X^*]=-\ad^*_{D_1X}$, for every $X\in\mg$. Moreover, $[D_4,\ad_X^*]=\ad_{D_1X}^*$, therefore
\[0=[D_4+\tran D_1,\ad_X^*],\]
which holds only if $\tran D_4+ D_1=:\phi$ is ad-invariant.
\end{proof}

Recall that the projection $P\colon\g\ltimes\g^*\to\g\ltimes\g^*$ onto $\g^*$ is a derivation, whose eigenspaces induce the grading~\eqref{eqn:grading_of_cotangent}.
\begin{proposition}
\label{prop:cotangentnik}
Let $\mg$ be a Lie algebra with Nikolayevsky derivation $\tilde N$; on the cotangent $\mg\ltimes\mg^*$, define $N=a(\tilde N-\tilde N^*+2P)$, where $a=\dim\mg/(2\dim\mg-\Tr \tilde N)$ satisfies $0<a<1$. Then $N$ is both a Nikolayevsky derivation and a metric Nikolayevsky derivation relative to the canonical metric.
\end{proposition}
\begin{proof}
Without loss of generality, we can assume that $\mg$ is irreducible.

Fix the grading~\eqref{eqn:grading_of_cotangent} on $\mg\ltimes \mg^*$. The derivation $N$ has degree zero. Every derivation $f$ of $\mg\ltimes \mg^*$ can be decomposed as the sum of graded derivations $f_{-1}+f_0+f_1$. It is clear that $\Tr(f_{\pm1})=0=\Tr(N f_{\pm1})$, so we only need to verify that $\Tr (f_0)=\Tr(N f_0)$. By Lemma~\ref{lm:D1234}, we can write $f_0=D_1-D_1^*+\lambda P + \phi^*$, where $D_1$ is a derivation of $\mg$ and $\phi\colon\g\to\g$ is ad-invariant and traceless.

We compute
\begin{gather*}
\Tr (ND_1)=a\Tr(\tilde ND_1)=a\Tr D_1;\\
\Tr (ND_1^*)=a\Tr(-\tilde N^* D_1^*+2D_1^*)=a\Tr D_1;\\
\Tr (NP)=a\Tr(-\tilde N^*P+2P)=a(-\Tr \tilde N+2\dim\g)=\dim\g=\Tr P.
\end{gather*}
It remains to show that
\[0=\Tr(N\phi^*)=a\Tr (-\tilde N^*\phi^*+2\phi^*)=-a\Tr(N\phi).
\]
To this end, let $\mg=\bigoplus \mb_t$ be the grading defined by the Nikolayevsky derivation $\tilde N$. Then, decomposing $\phi=\sum_k \phi_k$ according to degree, we have that $\Tr(N\phi_k)$ is zero whenever $k$ is nonzero. The zero-degree component $\phi_0$ is ad-invariant and traceless; since $\mg$ is irreducible, Proposition~\ref{prop:eigenvaluesofphi} implies that either $\phi_0$ only has $0$ as an eigenvalue, or it has  purely imaginary eigenvalues. Since $\phi_0$ preserves each $\mb_t$, this also applies to each restriction $\phi_0|_{\mb_t}\colon\mb_{t}\to\mb_t$. Therefore, \[\Tr(N\phi)=\Tr(N\phi_0)=\sum_t t \Tr \phi_0|_{\mb_t}=0.\]
This shows that $N$ is a Nikolayevsky derivation. Since $\tilde N-\tilde N^*$ is skew-symmetric and $P^*=\id-P$, it also lies in $\co$, so it is a metric Nikolayevsky derivation.

Arguing as in the proof of \cite[Theorem 1]{Nikolayevsky} and Theorem~\ref{thm:metricnik}, one  sees that the Nikolayevsky derivation $N$ corresponds to a vector $H$ such that $\langle[1],\tran FH\rangle = \sum d_i(\lambda_i-1)=\Tr N-\dim\mg$;
in addition, the matrix $FD^{-1}\tran F$ is positive-definite, so
$H=(FD^{-1}\tran F)^{-1}[1]$ satisfies
\[0<\langle [1],H\rangle =\langle -F[1],H\rangle  = \langle [1],\tran F H\rangle=\Tr N-\dim \g;\]
this implies both $a>0$ and $a<1$.
\end{proof}

Since both $\tilde N^*$ and $P$ are obtained by extending maps $\g^*\to\g^*$ to zero, we deduce the following:
\begin{corollary}
\label{cor:cotangentnegative}
Let $\g$ be a Lie algebra with Nikolayevsky derivation $\tilde N$ and let $N$ denote the metric Nikolayevsky derivation of the cotangent $\g\ltimes\g^*$ relative to the canonical ad-invariant metric.
If $\tilde N$ has zero as an eigenvalue (respectively, a negative eigenvalue), then $N$ has  zero as an eigenvalue (respectively, has a negative eigenvalue).
\end{corollary}

\begin{example}
The $9$-dimensional nilpotent Lie algebra $\g$ given by
\[(0,0,0,e^{12},e^{14},e^{15}+e^{23},-e^{34}+e^{16},-e^{35}+e^{17},e^{28}-e^{47}+e^{56}+e^{13})\]
has Nikolayevsky derivation
\[\diag\left(\frac{8}{27},-\frac{4}{27},\frac{8}{9},\frac{4}{27},\frac{4}{9},\frac{20}{27},\frac{28}{27},\frac{4}{3},\frac{32}{27}\right).\]
One can check explicitly that every symmetric ad-invariant $\phi\colon\g\to\g^*$ is zero on $\g'$, so $\g$ is \adsolitary by Remark~\ref{rk:WWadsolitary}. By Corollary~\ref{cor:cotangentnegative}, its cotangent $\mg\ltimes \mg^*$ has a solitary ad-invariant metric, and its (metric) Nikolayevsky derivation $N$ has both positive and negative eigenvalues. Explicitly, using Proposition~\ref{prop:cotangentnik} one obtains
$a=\frac{243}{326}$ and the (metric) Nikolayevsky derivation on $\mg\ltimes \mg^*$ is given by:
\[N=\frac{1}{163}\diag(36,-18,108,18,54,90,126,162,144,207,261,135,225,189,153,117,81,99).\]
\end{example}

\section{Double extensions and uniqueness}\label{sec:DE}

The double extension introduced by Medina and Revoy \cite{MedinaRevoy} (see also \cite{Ovandosurvey}) is an inductive procedure which allows to construct Lie algebras with ad-invariant metrics. The starting data is the following:
\begin{enumerate}
\item a metric Lie algebra $(\md, g_\md)$, with $g_\md$ ad-invariant;
\item a Lie  algebra $\mh$ and a representation $\pi\colon\mh\to \mathrm{Der}(\md)\cap \so(\md,g_\md)$.
\end{enumerate}
Given this data we can construct a Lie algebra admitting ad-invariant metrics as follows.

Let $\mh^*$ denote the dual space of $\mh$ and consider the vector space $\mq:=\mh\oplus\md\oplus\mh^*$ together with the Lie bracket
\begin{multline}\label{eq:debr}
[(H,X,\alpha),(H',X',\alpha')]_\mq\\
=([H,H']_\mh,[X,X']_\md+\pi(H)X'-\pi(H')X,\beta(X,X')+\ad_\mh^*(H)\alpha'-\ad_\mh^*(H')\alpha),
\end{multline}
where $\ad_\mh^*\colon\mh\to \gl(\mh^*)$ is the coadjoint representation and $\beta$ is given by
\begin{equation}\label{eq:debeta}
\beta(X,X')(H)=g_\md(\pi(H)X,X'),\qquad \text{ for all  }X,X'\in \md, \, H\in \mh.
\end{equation}
One can verify that $(\mq,[\;,\,])$ is a Lie algebra. Moreover, from~\eqref{eq:debr} one can see that $\md\oplus\mh^*$ is an ideal of $\mq$ and $\mq$ is the semidirect product of $\mh$ and $\md\oplus\mh^*$ by the action $\mh\ni H\to (\pi(H),\ad_\mh^*(H))\in \End(\md\oplus\mh^*)$.

For any (possibly degenerate) ad-invariant symmetric bilinear form $g_\mh$ on $\mh$, the following bilinear form on $\mq$ is nondegenerate:
\begin{equation}\label{eq:deme}
\tilde g((H,X,\alpha),(H',X',\alpha'))=g_\mh(H,H')+g_\md(X,X')+\alpha(H')+\alpha'(H);
\end{equation}
this defines an ad-invariant metric on $(\mq,[\;,\,])$.

Notice that when $\md=0$ and $g_\mh=0$, this construction gives the cotangent Lie algebra $T^*\mh$ with its canonical metric. From \cite{MeRe93} we know that every nonsimple irreducible Lie algebra admitting an ad-invariant metric is isometrically isomorphic to a double extension. Moreover, every solvable Lie algebra admitting an ad-invariant metric has nontrivial center and can be written as a double extension with $\dim \mh=1$ and $\md$ solvable (see also \cite{FavreSantharoubane:adInvariant}).

Note that even though~\eqref{eq:deme} defines a family of ad-invariant metrics on $\mq$, one for each choice of $g_\mh$, Proposition~\ref{prop:rigid_independent} implies that one of them is \rigid if and only if so is the metric defined by $g_\mh=0$.

Next, we are going to use the double extension procedure to show that solitary Lie algebras are necessarily solvable. We start by showing the following:

\begin{proposition}\label{pro:solitDE}
If $\mh$ is not \adsolitary, then any double extension of the form $\mh\oplus\md\oplus\mh^*$ is not \rigid.
\end{proposition}
\begin{proof}
The fact that $\mh$ is not \adsolitary implies that there is a linear map $\phi\colon\mh\to \mh^*$, ad-invariant and symmetric, which cannot be written as $D+D^t$, for any derivation $D\colon\mg\to\mg^*$.

Let $\mq=\mh\oplus\md\oplus\mh$ be a double extension of the Lie algebra $(\md,g_\md)$ by a representation \linebreak$\pi\colon\mh\to \Der(\md)\cap \so(\md,g_\md)$. One can easily check that the map $\tilde\phi\colon\mh\oplus\md\oplus\mh^*\to \mh\oplus\md\oplus\mh^*$, defined by $\tilde\phi(H,X,\alpha)=(0,0,\phi(H))$ for every $(H,X,\alpha)\in \mq$, is ad-invariant on $\mq$.

Assume for a contradiction that $\mq$ is \rigid. In particular the ad-invariant metric $\tilde g$ as in~\eqref{eq:deme} with $g_\mh=0$ is \rigid.

The map $\tilde\phi$  defined above is self-adjoint with respect to $\tilde g$. Indeed, since $\phi$ is symmetric, and thus verifies $\phi(H)(H')=\phi(H')(H)$ for every $H,H'\in \mh$, we have
\[
\tilde g(\tilde\phi(H,X,\alpha),(H',X',\alpha'))=\phi(H)(H')=\phi(H')(H)=\tilde g((H,X,\alpha),\tilde\phi(H',X',\alpha')),
\]
for all  $H,H'\in \mh$, $X,X'\in \md$ and $\alpha,\alpha'\in \mh^*$. Therefore, $\tilde\phi$ is ad-invariant and $\tilde g$-self-adjoint on $\mq$. By our assumption of $\mq$ being \rigid, there exists a derivation $ E$ of $\mq$ such that $\tilde\phi=E+ E^*$. In particular, for every $H,H'\in \mh$, $E$ verifies
\begin{equation}\label{eq:Eproy}
\mathrm{Pr}_{\mh^*}\circ E[H,H']=\ad_H^*(\mathrm{Pr}_{\mh^*}\circ EH')-\ad_Y^*(\mathrm{Pr}_{\mh^*}\circ EH),
\end{equation}
where $\mathrm{Pr}$ denotes the projection map.
This implies that the linear map $D\colon\mh \to\mh^*$ defined by $D= \Pr_{\mh^*}\circ E|_{\mh}$ on $\mh$ is a derivation of $T^*\mh$. In addition, since $\tilde g(\mh,\mh)=0$, for every $H,H'\in \mh$ we have
\begin{align*}
(D+D^t)(H)(H')&=\mathrm{Pr}_{\mh^*}( EH')(H)+\mathrm{Pr}_{\mh^*}(EH)(H')\\
& =\tilde g(H,EH')+\tilde g(EH,H')=\tilde g(\tilde \phi H,H')=\phi (H)(H').
\end{align*}
Therefore, $D+D^t=\phi$, which is a contradiction.
\end{proof}

\begin{remark}
If $\mh$ is \adsolitary and $(\md,g_{\md})$ is \rigid, it could still be the case that $\mh\oplus\md\oplus\mh^*$ is not \rigid. For example, the $12$-dimensional Lie algebra in  Example~\ref{ex:CharacteristicallyNilpotentDIM12} is the double extension by $\R$ of a $10$-dimensional metric nilpotent Lie algebra, which is \rigid.
\end{remark}

According to \cite[Theorem I]{MedinaRevoy}, every Lie algebra $\mg$ admitting an ad-invariant metric, which is irreducible as a metric Lie algebra, is either simple or it can be written as a double extension by a simple or a $1$-dimensional Lie algebra. More precisely, it is shown there that given a maximal ideal $\mm$ of $\mg$, the data to construct the double extension are $\mh=\mg/\mm$ and $\md=\mm/\mm^\bot$, and $\mh$ is either simple or it is of dimension $1$.

\begin{theorem}
\label{thm:solitaryimpliessolvable}
If a Lie algebra  is $T^*$-solitary, then it is solvable.
\end{theorem}
\begin{proof}
Let $\mg$ be a $T^*$-solitary Lie algebra. We first prove the statement for Lie algebras admitting an ad-invariant metric.

Suppose that $\mg$ admits an ad-invariant metric $g$, which is solitary by Theorem~\ref{thm:rigiffTrig}.

Assume first that $\mg$ is irreducible, so that $(\mg,g)$ is irreducible as a metric Lie algebra by Proposition~\ref{prop:irr_is_irr}. If $\mg$ is not solvable, then by Levi's theorem we can write $\mg=\ms\ltimes \mm$, where $\ms$ a simple Lie algebra and $\mm$ is a maximal ideal not necessarily solvable. Following \cite{MedinaRevoy}, $\mg$ is isometrically isomorphic to a double extension of $\md=\mm/\mm^\bot$, which admits an ad-invariant metric, by
$\mh=\mg/\mm=\ms$. We may write $\mg=\mh\oplus\md\oplus\mh^*$.

Since $\mh$ is simple, $\mh$ is not $T^*$-\rigid by Corollary~\ref{cor:simplenotTsol} and thus $\mg$ is not \rigid by Proposition~\ref{pro:solitDE}, which is absurd. So $\mg$ is solvable if it is irreducible.

Assume now that $\mg$ is not irreducible. Then it is reducible and each irreducible factor is \rigid by Proposition~\ref{prop:decomposablesolitary}. So each of these factors is solvable by the reasoning above. Therefore, $\mg$ is solvable too.

In the general case, $\g$ being $T^*$-solitary implies that $T^*\g$ (which admits an ad-invariant metric) is solitary by Corollary~\ref{cor:cotangentunique}. So we have that $T^*\mg$ is solvable and, being isomorphic to the quotient of $T^*\g$ by the ideal $\g^*$, the Lie algebra $\g$ is also solvable.
\end{proof}
\begin{remark}
The examples of nonsolitary solvable Lie algebras with ad-invariant metrics presented in Section~\ref{sec:uniqueness} show that the converse of Theorem~\ref{thm:solitaryimpliessolvable} does not hold.
\end{remark}

Finally, we prove that the metric Nikolayevsky derivation provides a decomposition of the Lie algebra as a double extension.

Let $(\mg,g)$ be a Lie algebra endowed with an ad-invariant metric such that a metric Nikolayevsky derivation has nonnegative integer eigenvalues. By Remark~\ref{rem:nilpgrad}, $\mg$ is isomorphic as a Lie algebra to $\mh\ltimes \mn$, where $\mh$ is the eigenspace corresponding to the zero eigenvalue of the metric Nikolayevsky derivation, and $\mn$ is the direct sum of the eigenspaces corresponding to the positive eigenvalues.
The following result describes the structure of $\mg$ as a metric Lie algebra.

\begin{proposition}
\label{prop:nonnegativeimpliesextension}
Let $\mg$ be a Lie algebra endowed with an ad-inva\-riant metric $g$ such that the metric Nikolayevsky derivation $N$ is nonzero, singular and with nonnegative eigenvalues. Then $\mg$ is isometric to a double extension $\mh\oplus \md\oplus\mh^*$ with $\mh=\ker N$ and $g_\mh=0$.
\end{proposition}

\begin{proof}
Since the metric Nikolayevsky derivation is an element of $\co$ with positive trace, there exists $l>0$ such that if $\lambda$ is an eigenvalue of the metric Nikolayevsky derivation, then $l-\lambda$ is an eigenvalue too.

Let $\mg=\bigoplus_t \mb_{t}$ be the grading given by the eigenspaces of the metric Nikolayevsky derivation; the index $t$ ranges between $0$ and $l$, for some positive $l$. By Remark~\ref{rem:nilpgrad}, $\mh:=\mb_0$ is a subalgebra and $\mn:=\bigoplus_{t>0}\mb_{t}$ is an ideal of $\mg$. Moreover, $\mb_l$ is in the center of $\mn$ so we can consider the quotient Lie algebra $\md:=\mn/\mb_l$. The metric $g$ induces a nondegenerate ad-invariant metric $g_\md$ so that the quotient map $q\colon\mathfrak n\to \md$ is an isometry. Moreover, $q$ is injective when restricted to $\bigoplus_{0<t<l}\mb_{t}$.

Since $\mb_{i}$ and $\mb_{j}$ are orthogonal unless $i+j=l$, $\md$ inherits an ad-invariant metric $g_{\md}$, and $\mb_l$ can be identified with $\mh^*$ through  the linear map $\eta\colon\mb_l\to \mh^*$, $x\mapsto \eta(x)=g(x,\cdot)$.

For any $X\in \mh$, $\ad_X$ preserves $\mb_{t}$, for $t=0, \ldots, l$, so $\ad_X$ also preserves $\mathfrak n$ and it defines a skew-symmetric map $\overline{ \ad_X}$ on $\md$. Moreover, $\pi\colon\mh\to \mathrm{Der}(\mg)\cap \so(\md, g_{\md})$, $X\mapsto \pi(X)=\overline{ \ad_X}$ is a representation.

Let $Y_0,Y_0\in \mb_0$ and $Y,Y'\in  \bigoplus_{0<t<l}\mb_{t}$. Then
\[
\pi(Y_0)q(Y)=\overline{\ad_{Y_0}}q(Y)=q([Y_0,Y])
\]
and $\beta$ in~\eqref{eq:debeta} is given by,
\[
\beta(q(Y),q(Y'))(H)=g_\md(\pi(H)q(Y),q(Y'))=g_\md(q([H,Y]),q(Y'))=g([H,Y],Y'),
\]
for every $H\in \mb_0$.
This implies
\begin{equation}\label{eq:betaisom}
\beta(q(Y),q(Y'))=\eta( \mathrm{Pr}_{\mb_l}(\ad_YY')).
\end{equation}
Finally, by definition of $\eta$ we have
\begin{equation}\label{eq:adjointisom}
\ad_\mh^*(Y_0)\eta(Y_l)=\eta([Y_0,Y_l]), \text{ for every } Y_0\in \mb_{0}, Y_l\in \mb_{l}.
\end{equation}

Let $(\mq,\tilde g)$ be the double extension of $(\md, g_{\md})$  by  $\mh$, taking the trivial metric on $\mh$, and the Lie algebra homomorphism $\pi$. The Lie bracket in $\mq$ is given by~\eqref{eq:debr}, and the ad-invariant metric $\tilde g$ is as in~\eqref{eq:deme}, with $g_\mh=0$.
We shall prove that $\mg$ is isometrically isomorphic to $\mq$.

Since $\mg$ is graded, every element in $\mg$ can be written as $Y_0+Y+Y_l$, with $Y_0\in \mb_0$, $Y_l\in \mb_l$ and $Y\in \bigoplus_{0<t<l}\mb_{t}$. This allows us to define the linear map
\[
\zeta\colon\mg\to\mq,\qquad \zeta(Y_0+Y+Y_l)=(Y_0,q(Y),\eta(Y_l)).
\]

From the fact that $\mb_{i}\bot \mb_{j}$ unless $i+j=l$, one can easily check that $\zeta$ is an isometry.
Moreover, by~\eqref{eq:debr},~\eqref{eq:betaisom} and~\eqref{eq:adjointisom} one obtains that $\zeta$ is a Lie algebra homomorphism, which is clearly bijective.
\end{proof}

Under the hypotheses of Proposition~\ref{prop:nonnegativeimpliesextension}, and due to Proposition~\ref{prop:impliesadsolitary}, $\g$ is \rigid if and only if the zero eigenspace $\mh$ is \adsolitary.

\section{The solitary condition in low dimension}\label{sec:lowdim}
In this last section we apply the results obtained along the paper to prove uniqueness of the ad-invariant metric, in the sense of Theorem~\ref{thm:rigidimpliesuniqueness}, for solvable Lie algebras of dimension $\leq 6$ and real nilpotent Lie algebras of dimension $\leq 10$. Indeed, we prove that every Lie algebra in these two classes is $T^*$-solitary, regardless of whether an ad-invariant metric exists; when it does, it is necessarily solitary by Theorem~\ref{thm:rigiffTrig}.

To this purpose, we start by giving in the following lemma the classification of solvable Lie algebras admitting ad-invariant metrics, up to dimension $4$. The result can be proved following similar techniques as in Propositions~2.3 and~2.5, and Lemma~2.4 in \cite{dBOV}, and by computing skew-symmetric derivations.

\begin{lemma}\label{lm:solv4dim}
Let $\g$ be a solvable Lie algebra over $\K=\R,\C$ admitting an ad-invariant metric $g$, and assume $\dim\g\leq 4$. If $\mg$ is nonabelian, then $\g$ has dimension $4$, $\Der(\mg)\cap \so(\mg,g)=\ad(\mg)$, and there is an isometric isomorphism mapping $(\g,g)$ to one of
\begin{enumerate}
\item $\md_1=(0,13,-12,-23)$ with the ad-invariant metric $g_1=e^1\odot e^4+e^2\otimes e^2+e^3\otimes e^3$; or
\item $\md_2=(0,-12,13,-23)$ with the ad-invariant metric $g_2=e^1\odot e^4+e^2\odot e^3$.
\end{enumerate}
If $\K=\C$, the metric Lie algebras $(\md_1,g_1)$ and $(\md_2,g_2)$ are isometrically isomorphic.
\end{lemma}

Now we proceed with a structural result.

\begin{proposition}\label{pro:dim6solv}
Every solvable Lie algebra over $\K=\R,\C$ of dimension $\leq 6$ admitting ad-invariant metrics is either an orthogonal direct sum of $\K^2\oplus \md$, with $\md$ solvable, or isometric to a double extension of the form $\mh\oplus\md\oplus\mh^*$ with $\md$ abelian and $\mh$ of dimension $1$.
\end{proposition}

\begin{proof}
Let $\mg$ be a solvable Lie algebra of dimension $\leq 6$ and let $g$ be an ad-invariant metric on $\mg$. Since $\mz(\mg)\neq 0$, $(\mg,g)$ is isometric to a double extension of the form $\mg=\mh\oplus \md \oplus\mh^*$, where $(\md,g_\md)$ is solvable, $\mh=\K e$ and the representation $\pi\colon\mh\to \Der(\md)\cap \so(\md,g_\md)$ is determined by a skew-symmetric derivation  $D:=\pi(e)$ of $(\md,g_\md)$ (see Section~\ref{sec:DE}). Recall that, if $f$ denotes a generator of $\mh^*$, the Lie bracket of $\mg$ satisfies
\begin{equation}\label{eq:br}
[e,x]=Dx,\qquad [x,y]=[x,y]_{\md}+g_\md(Dx,y)f,\quad \text{ for all }x,y\in \md,
\end{equation}
and for every $x,y\in \md$,  the metric verifies $g(x,y)=g_\md(x,y)$, $g(e,x)=g(f,x)=0$ and $g(e, f)=1$.

If $D=0$ then $\mg$ is the orthogonal direct sum of $\K^2\simeq \mh\oplus\mh^*$ and $\md$. Moreover, if $\md$ is abelian, then the result follows trivially.

So from now on we assume that $\md$ is not abelian and  $D\neq 0$. By Lemma~\ref{lm:solv4dim}, $\dim \md=4$ and $D=\ad_w$ for some $w\in \md$. By \cite[Proposition 2.11]{FavreSantharoubane:adInvariant} we obtain that
$\mg$ is isometrically isomorphic to a double extension with zero derivation. Therefore, $\mg$ can be written as an orthogonal direct sum $\K^2\oplus\md$.
\end{proof}

Now we are ready to introduce the main results of the section.
\begin{proposition}\label{prop:solvablesolitary}
Every solvable Lie algebra of dimension $\leq 6$ is \adsolitary. In particular, every ad-invariant metric on a solvable Lie algebra of dimension $\leq 6$ is solitary.
\end{proposition}

\begin{proof}
We first prove the second part of the statement.

By Proposition~\ref{pro:dim6solv}, every solvable Lie algebra $\mg$ of dimension $\leq 6$ admitting ad-invariant metrics  is either an orthogonal direct sum $\K^2\oplus \md$ with $\md$ solvable or isometric to a double extension of an abelian Lie algebra.

Suppose that $\mg$ can be written as an orthogonal direct sum $\K^2\oplus \md$. Then $\md$ is either abelian or 4-dimensional and listed in Lemma~\ref{lm:solv4dim}. Straightforward computations show that $\md$ is solitary in either case. Since $\K^2$ is solitary, $\mg$ is solitary too by Proposition~\ref{prop:decomposablesolitary}.

Suppose that $\mg$ is isometric to a double extension $(\mh\oplus\md\oplus\mh^*,g)$, where $\md$ is abelian and $\mh$ is $1$-dimensional. Consider a complementary subspace $\mb_1$ of $\mh^*$  inside the ideal $\md\oplus\mh^*$, and denote $\mb_0=\mh$, $\mb_2=\mh^*$. It is easy to check that $\mg=\mb_0\oplus\mb_1\oplus\mb_2$ is a grading of $\mg$ which verifies~\eqref{eqn:co_gradation} with respect to $g$ and taking $l=2$. Furthermore, since $\dim\mh=1$, every map $\mb_0\to\mb_2$ is a derivation, so Lemma~\ref{lemma:uniquenesslemma} applies.

Now consider an arbitrary solvable Lie algebra $\g$ of dimension $\leq 6$. By Proposition~\ref{prop:quotientweak}, a solvable Lie algebra of dimension $\leq6$ which is not $T^*$-solitary has a quotient which has a nonsolitary ad-invariant metric, which is also a solvable Lie algebra of dimension $\leq6$; by the first part of the proof, this is absurd.
\end{proof}

Real nilpotent Lie algebras of dimension $\leq 10$ admitting ad-invariant metrics are classified by Kath in \cite{Kath:NilpotentSmallDim}. It is easy to show that all these Lie algebras are actually nice. Indeed, we report each one of them in Table~\ref{tbl:kath} by giving the structure coefficients with respect to a nice basis, and the parameters for their construction as listed by Kath; notice that we use the updated version of  \cite[Theorem 1.1]{Kath:NilpotentSmallDim} available at \texttt{arXiv:math/0505688v2}. In addition, one can check that in every case the ad-invariant metric is $\sigma$-diagonal.
The bases we give are obtained from Kath's list by reordering the elements, except in the cases~\ref{Numb:Table1Case4a24}) and~\ref{Numb:Table1Case4b31}), where a more complicated change of basis is needed.

Since Kath classifies metric Lie algebras, whilst Table~\ref{tbl:kath} only lists Lie algebras, a single entry $\g$ in our table may correspond to more entries in Kath's list.

{\setlength{\tabcolsep}{2pt}
\begin{table}[thp]
\centering
{ \caption{\label{tbl:kath} Nilpotent Lie algebras of dimension $\leq 10$ admitting an ad-invariant metric, as classified by Kath in \cite{Kath:NilpotentSmallDim}, written in terms of a nice basis such that the ad-invariant metric is $\sigma$-diagonal.}
\begin{small}
\begin{tabular}{ >{\stepcounter{rowno}\therowno)}r  L c L}
\toprule
\multicolumn{1}{r}{} & \textnormal{Lie algebra } \g & \multicolumn{2}{l}{Cases in Kath's list} \\
\midrule
& 0,0,e^{25},-e^{15},e^{12} & 8 &  \\ 
& 0,0,0,e^{23},-e^{13},e^{12} & 7(a) &  \\ 
& 0,0,e^{12},e^{26}+e^{37},-e^{16},-e^{17},e^{13} & 6(a) &  \\ 
& 0,0,e^{12},e^{26}+e^{37},-e^{16}+e^{38},-e^{17}-e^{28},e^{13},e^{23} & 6(b) & \lie{a}\in\{\R^2,\R^{0,2}\} \\ 
& 0,0,e^{12},e^{26}-e^{37},-e^{16}+e^{38},e^{17}-e^{28},e^{13},e^{23} & 6(b) & \lie{a}=\R^{1,1} \\
& 0,0,0,e^{27}+e^{38},-e^{17},-e^{18},e^{12},e^{13} & 7(b) & \lie{a}\in\{\R^2,\R^{0,2}\} \\
& 0,0,0,-e^{27}+e^{38},e^{17},-e^{18},e^{12},e^{13} & 7(b) & \lie{a}=\R^{1,1} \\
& 0,0,e^{14},e^{12},e^{47}+e^{28}+e^{39},-e^{18},-e^{19},-e^{17},e^{13} & 3(a) & \gamma=0 \\
\multicolumn{1}{r}{\multirow{2}{*}{\stepcounter{rowno}\therowno)}} & \multicolumn{1}{L}{0,0,e^{14},e^{12},e^{47}+e^{28}+e^{39},} & \multirow{2}{*}{3(a)} & \multirow{2}{*}{$\gamma=\sigma^2\wedge\sigma^Y\wedge\sigma^Z$} \\
\multicolumn{1}{r}{} &
\multicolumn{1}{R}{-e^{18}+e^{34},-e^{19}-e^{24},-e^{17}+e^{23},e^{13}} & & \\
& 0,0,e^{12},0,e^{27}+e^{39},-e^{17}+e^{34},-e^{19}-e^{24},e^{23},e^{13} & 4(c) & \\
& 0,0,0,0,e^{39},e^{34},-e^{19}-e^{24},e^{23},e^{13} & 5(c) & \\
& 0,0,0,e^{27}+e^{38},-e^{17}+e^{39},-e^{18}-e^{29},e^{12},e^{13},e^{23} & 7(c) & \lie{a}=\R^3 \text{ or } \R^{0,3} \\
& 0,0,0,-e^{27}+e^{38},e^{17}+e^{39},-e^{18}-e^{29},e^{12},e^{13},e^{23} & 7(c) & \lie{a}=\R^{2,1} \text{ or } \R^{1,2} \\
& 0,0,0,0,0,e^{25},-e^{15},e^{45},-e^{35},e^{12}+e^{34} & 1 & \\
& 0,0,0,e^{12},e^{13},e^{29}+e^{3,10}+e^{45},-e^{19},-e^{1,10},-e^{15},e^{14} & 2 & \gamma=\sigma^1\wedge\sigma^Y\wedge\sigma^Z \\
\multicolumn{1}{r}{\multirow{2}{*}{\stepcounter{rowno}\therowno)}} & \multicolumn{1}{L}{0,0,0,e^{12},e^{13},} & \multirow{2}{*}{2} & \multirow{2}{*}{$\gamma=\sigma^1\!\wedge\!\sigma^Y\!\wedge\!\sigma^Z\!+\!\sigma^2\!\wedge\!\sigma^3\!\wedge\!\sigma^Z$} \\
\multicolumn{1}{r}{} & \multicolumn{1}{R}{e^{29}+e^{3,10}+e^{45},-e^{19}+e^{35},-e^{1,10}-e^{25},-e^{15},e^{14}+e^{23}} & & \\
\multicolumn{1}{r}{\multirow{2}{*}{\stepcounter{rowno}\therowno)}} & \multicolumn{1}{L}{0,0,e^{14},e^{12},e^{47}+e^{28}+e^{39},} & \multirow{2}{*}{3(b)} & \\
\multicolumn{1}{r}{} & \multicolumn{1}{R}{-e^{18}+e^{4,10},-e^{19},-e^{17}-e^{2,10},e^{13},e^{24}} & & \\
\multicolumn{1}{r}{\multirow{2}{*}{\stepcounter{rowno}\therowno)}} & \multicolumn{1}{L}{0,0,e^{14},e^{12},e^{47}+e^{28}-e^{39},} & \multirow{2}{*}{3(c)} & \multicolumn{1}{L}{\alpha\!\in\!\{\sigma^1\!\wedge\!\sigma^Y\!\otimes\! A_1\!+\!\sigma^2\!\wedge\!\sigma^Z\!\otimes\! A_2,}
\\
\multicolumn{1}{r}{} & \multicolumn{1}{R}{-e^{18}+e^{4,10},e^{19},-e^{17}-e^{2,10},e^{13},e^{24}} & & \multicolumn{1}{R}{\sigma^2\!\wedge\!\sigma^Z\!\otimes\! A_1\!+\!\sigma^1\!\wedge\!\sigma^Y\!\otimes\! A_2\}}\\
\multicolumn{1}{r}{\multirow{2}{*}{\stepcounter{rowno}\therowno)}} & \multicolumn{1}{L}{0,0,e^{12},0,e^{27}+e^{39}+e^{4,10},-e^{17}+e^{49}+e^{3,10},} & \multirow{2}{*}{4(a)} & \multirow{2}{*}{$(\alpha, \gamma) = (\alpha_1,0)$} \\
\multicolumn{1}{r}{} & \multicolumn{1}{R}{-e^{19}-e^{2,10},-e^{29}-e^{1,10},e^{13}+e^{24},e^{14}+e^{23}} & & \\
& 0,0,e^{12},0,e^{27}+e^{39}+e^{4,10},-e^{17},-e^{19},-e^{1,10},e^{13},e^{14} & 4(a) & (\alpha,\gamma)=(\alpha_5,0) \\
\multicolumn{1}{r}{\multirow{2}{*}{\stepcounter{rowno}\therowno)}} & \multicolumn{1}{L}{0,0,e^{12},0,e^{27}+e^{39}+e^{4,10},} & \multirow{2}{*}{4(a)} & \multirow{2}{*}{$(\alpha,\gamma)=(\alpha_5,\gamma_0)$} \\
\multicolumn{1}{r}{} & \multicolumn{1}{R}{-e^{17}+e^{34},-e^{19}-e^{24},e^{23}-e^{1,10},e^{13},e^{14}} & & \\
& 0,0,e^{12},0,e^{27}+e^{39},-e^{17}+e^{4,10},-e^{19},-e^{2,10},e^{13},e^{24} & 4(a) & (\alpha,\gamma)=(\alpha_6,0)\\
\multicolumn{1}{r}{\refstepcounter{rowno}\therowno\label{Numb:Table1Case4a24})} & 0,0,0,e^{12},e^{13},e^{24},e^{15}, e^{26},e^{28}-e^{35},e^{18}+e^{46} & 4(a) & (\alpha,\gamma)=(\alpha_6,\gamma_0) \\
\multicolumn{1}{r}{\multirow{2}{*}{\stepcounter{rowno}\therowno)}} & \multicolumn{1}{L}{0,0,e^{12},0,e^{27}+e^{3,10}+e^{49},} & \multirow{2}{*}{4(b)} & \multirow{2}{*}{$(\alpha,\gamma)=(\alpha_1,0)$} \\
\multicolumn{1}{r}{} & \multicolumn{1}{R}{-e^{17}+e^{4,10}+e^{39},-e^{1,10}-e^{29},-e^{2,10}-e^{19},e^{13}+e^{24},e^{14}+e^{23}} & & \\
\multicolumn{1}{r}{\multirow{2}{*}{\stepcounter{rowno}\therowno)}} & \multicolumn{1}{L}{0,0,e^{12},0,e^{27}+e^{3,10}+e^{49},} & \multirow{2}{*}{4(b)} & \multirow{2}{*}{$(\alpha,\gamma)=(\alpha_2,0)$} \\
\multicolumn{1}{r}{} & \multicolumn{1}{R}{-e^{17}-e^{4,10}+e^{39},-e^{1,10}-e^{29},e^{2,10}-e^{19},e^{13}-e^{24},e^{14}+e^{23}} & & \\
\multicolumn{1}{r}{\multirow{2}{*}{\stepcounter{rowno}\therowno)}} & \multicolumn{1}{L}{0,0,e^{12},0,e^{27}+e^{3,10}+e^{49},} & \multirow{2}{*}{4(b)} & \multirow{2}{*}{$(\alpha,\gamma)=(\alpha_3,0)$} \\
\multicolumn{1}{r}{} & \multicolumn{1}{R}{-e^{17}+e^{39},-e^{1,10}-e^{29},-e^{19},e^{13},e^{14}+e^{23}} & & \\
& 0,0,e^{12},0,e^{27}-e^{39}+e^{4,10},-e^{17},e^{19},-e^{1,10},e^{13},e^{14} & 4(b) & (\alpha,\gamma)\in\{(\alpha_5,0),(\alpha_5',0)\} \\
\multicolumn{1}{r}{\multirow{2}{*}{\stepcounter{rowno}\therowno)}} & \multicolumn{1}{L}{0,0,e^{12},0,e^{27}-e^{39}+e^{4,10},} & \multirow{2}{*}{4(b)} & \multirow{2}{*}{$(\alpha,\gamma)\in\{(\alpha_5,\gamma_0),(\alpha_5',\gamma_0)\}$} \\
\multicolumn{1}{r}{} & \multicolumn{1}{R}{-e^{17}+e^{34},e^{19}-e^{24},-e^{1,10}+e^{23},e^{13},e^{14}} & & \\
& 0,0,e^{12},0,e^{27}-e^{39},-e^{17}+e^{4,10},e^{19},-e^{2,10},e^{13},e^{24} & 4(b) & (\alpha,\gamma)\in\{(\alpha_6,0),(\alpha_6',0)\} \\
\multicolumn{1}{r}{\refstepcounter{rowno}\therowno\label{Numb:Table1Case4b31})} & 0,0,0,e^{12},e^{13},e^{24},e^{15}, e^{26},e^{28}+e^{35},e^{18}+e^{46} & 4(b) & (\alpha,\gamma)\in\{(\alpha_6,\gamma_0),(\alpha_6',\gamma_0)\} \\
\multicolumn{1}{r}{\multirow{2}{*}{\stepcounter{rowno}\therowno)}} & \multicolumn{1}{L}{0,0,0,0,e^{39}+e^{4,10},} & \multirow{2}{*}{5(a)} & \multirow{2}{*}{$\alpha=\alpha_1$} \\
\multicolumn{1}{r}{} & \multicolumn{1}{R}{e^{49}+e^{3,10},-e^{19}-e^{2,10},-e^{29}-e^{1,10},e^{13}+e^{24},e^{14}+e^{23}} & & \\
&
0,0,0,0,e^{39}+e^{24},e^{3,10}-e^{14},-e^{19}-e^{2,10},e^{12},e^{13},e^{23} & 5(a) & \alpha=\alpha_4 \\
\multicolumn{1}{r}{\multirow{2}{*}{\stepcounter{rowno}\therowno)}} & \multicolumn{1}{L}{0,0,0,0,e^{3,10}+e^{39},e^{4,10}+e^{49},} & \multirow{2}{*}{5(b)} & \multirow{2}{*}{$\alpha=\alpha_1$} \\
\multicolumn{1}{r}{} & \multicolumn{1}{R}{-e^{1,10}-e^{29},-e^{2,10}-e^{19},e^{13}+e^{24},e^{23}+e^{14}}\\
\multicolumn{1}{r}{\multirow{2}{*}{\stepcounter{rowno}\therowno)}} & \multicolumn{1}{L}{0,0,0,0,e^{3,10}+e^{49},} & \multirow{2}{*}{5(b)} & \multirow{2}{*}{$\alpha=\alpha_2$} \\
\multicolumn{1}{r}{} & \multicolumn{1}{R}{e^{39}-e^{4,10},-e^{1,10}-e^{29},e^{2,10}-e^{19},e^{13}-e^{24},e^{14}+e^{23}} & & \\
& 0,0,0,0,e^{3,10}+e^{49},e^{39},-e^{1,10}-e^{29},-e^{19},e^{13},e^{23}+e^{14} & 5(b) & \alpha=\alpha_3 \\
& 0,0,0,0,e^{3,10}+e^{24},e^{39}-e^{14},-e^{1,10}-e^{29},e^{12},e^{13},e^{23} & 5(b) & \alpha=\alpha_4 \\
\bottomrule
\end{tabular}
\end{small}
}
\end{table}
}

\begin{proposition}\label{cor:nilpleq10solitary}
Every real nilpotent Lie algebra of dimension $\leq 10$ is \adsolitary. In particular, any ad-invariant metric on a real nilpotent Lie algebra of dimension $\leq10$ is \rigid.
\end{proposition}
\begin{proof}
Arguing as in the proof of Proposition~\ref{prop:solvablesolitary}, we only need to prove the second part of the statement and apply  Proposition~\ref{prop:quotientweak} to complete the proof.

Every nilpotent Lie algebra of dimension $\leq 10$ with an ad-invariant metric has a nice basis such that the metric is $\sigma$-diagonal, as listed in Table~\ref{tbl:kath}. By Corollary~\ref{cor:metricNikIsNik}, the metric Nikolayevsky derivation coincides with the Nikolayevsky derivation in each case. On a nice Lie algebra, computing the Nikolayevsky derivation is straightforward, and we see that in each case all the eigenvalues of the Nikolayevsky derivation are positive, so Theorem~\ref{thm:uniquenesstheorem} applies.
\end{proof}

\begin{remark}
We do not know whether there exist nilpotent Lie algebras of dimension $11$ which are not $T^*$-solitary. However, Example~\ref{ex:CharacteristicallyNilpotentDIM12} shows that not all $12$-dimensional nilpotent Lie algebras are $T^*$-solitary.
\end{remark}

\FloatBarrier

\bibliographystyle{plain}

\bibliography{UniquenessAd.bib}

\begin{thebibliography}{10}

\bibitem{AC01}
J.~M. {Ancochea} and R.~{Campoamor}.
\newblock {Characteristically nilpotent Lie algebras: A survey}.
\newblock {\em {Extracta Math.}}, 16(2):153--210, 2001.

\bibitem{As79}
V.~V. {Astrakhantsev}.
\newblock {Decomposability of metrizable Lie algebras}.
\newblock {\em {Funct. Anal. Appl.}}, 12:210--212, 1979.

\bibitem{BaBe07}
I.~{Bajo} and S.~{Benayadi}.
\newblock {Lie algebras with quadratic dimension equal to 2}.
\newblock {\em {J. Pure Appl. Algebra}}, 209(3):725--737, 2007.

\bibitem{Bagl}
O.~{Baues} and W.~{Globke}.
\newblock {Rigidity of compact pseudo-Riemannian homogeneous spaces for
  solvable Lie groups}.
\newblock {\em {Int. Math. Res. Not.}}, 2018(10):3199--3223, 2018.

\bibitem{BaumKath}
H.~Baum and I.~Kath.
\newblock Doubly extended {L}ie groups---curvature, holonomy and parallel
  spinors.
\newblock {\em Differential Geom. Appl.}, 19(3):253--280, 2003.

\bibitem{Bo97}
M.~Bordemann.
\newblock Nondegenerate invariant bilinear forms on nonassociative algebras.
\newblock {\em Acta Math. Univ. Comenian. (N.S.)}, 66(2):151--201, 1997.

\bibitem{ContiDelBarcoRossi}
D.~Conti, V.~del Barco, and F.~A. Rossi.
\newblock {D}iagram involutions and homogeneous {R}icci-flat metrics.
\newblock {\em Manuscripta Math.}, 2020.

\bibitem{ContiRossi:Construction}
D.~Conti and F.~A. Rossi.
\newblock {C}onstruction of nice nilpotent {L}ie groups.
\newblock {\em J. Algebra}, 525:311--340, 2019.

\bibitem{ContiRossi:RicciFlat}
D.~Conti and F.~A. Rossi.
\newblock Ricci-flat and {E}instein pseudoriemannian nilmanifolds.
\newblock {\em Complex Manifolds}, 6(1):170--193, 2019.

\bibitem{dBO12}
V.~{del Barco} and G.~P. {Ovando}.
\newblock {Free nilpotent Lie algebras admitting ad-invariant metrics}.
\newblock {\em {J. Algebra}}, 366:205--216, 2012.

\bibitem{dBOV}
V.~{del Barco}, G.~P. {Ovando}, and F.~{Vittone}.
\newblock {On the isometry groups of invariant Lorentzian metrics on the
  Heisenberg group}.
\newblock {\em {Mediterr. J. Math.}}, 11(1):137--153, 2014.

\bibitem{FavreSantharoubane:adInvariant}
G.~Favre and L.~J. Santharoubane.
\newblock Symmetric, invariant, nondegenerate bilinear form on a {L}ie algebra.
\newblock {\em J. Algebra}, 105(2):451--464, 1987.

\bibitem{FiSt96}
J.~M. {Figueroa-O'Farrill} and S.~{Stanciu}.
\newblock {On the structure of symmetric self-dual Lie algebras}.
\newblock {\em {J. Math. Phys.}}, 37(8):4121--4134, 1996.

\bibitem{FisherGrayHydon}
D.~J. Fisher, R.~J. Gray, and P.~E. Hydon.
\newblock Automorphisms of real {L}ie algebras of dimension five or less.
\newblock {\em J. Phys. A}, 46(22):225204, 18, 2013.

\bibitem{Humphreys}
J.~E. Humphreys.
\newblock {\em Introduction to {L}ie algebras and representation theory},
  volume~9 of {\em Graduate Texts in Mathematics}.
\newblock Springer-Verlag, New York-Berlin, 1978.
\newblock Second printing, revised.

\bibitem{Karpi90}
G.~Karpilovsky.
\newblock {\em Frobenius and symmetric algebras}.
\newblock Springer Netherlands, Dordrecht, 1990.

\bibitem{Kath:NilpotentSmallDim}
I.~Kath.
\newblock Nilpotent metric {L}ie algebras of small dimension.
\newblock {\em J. Lie Theory}, 17(1):41--61, 2007.

\bibitem{Ka20}
I.~Kath.
\newblock {Existence of Cocompact Lattices in Lie Groups With a Bi-invariant
  Metric of Index 2}.
\newblock {\em {Int. Math. Res. Not.}}, 03 2020.

\bibitem{KathOlbrich:Metric}
I.~Kath and M.~Olbrich.
\newblock Metric {L}ie algebras with maximal isotropic centre.
\newblock {\em Math. Z.}, 246(1-2):23--53, 2004.

\bibitem{KO08}
I.~{Kath} and M.~{Olbrich}.
\newblock {The classification problem for pseudo-Riemannian symmetric spaces}.
\newblock In {\em {Recent developments in pseudo-Riemannian geometry}}, pages
  1--52. Z\"{u}rich: European Mathematical Society, 2008.

\bibitem{Kos56}
B.~{Kostant}.
\newblock {On differential geometry and homogeneous spaces. II}.
\newblock {\em {Proc. Natl. Acad. Sci. USA}}, 42:354--357, 1956.

\bibitem{LauretWill:EinsteinSolvmanifolds}
J.~Lauret and C.~Will.
\newblock Einstein solvmanifolds: existence and non-existence questions.
\newblock {\em Math. Ann.}, 350(1):199--225, 2011.

\bibitem{MedinaRevoy}
A.~Medina and P.~Revoy.
\newblock {Alg\`{e}bres de Lie et produit scalaire invariant.}
\newblock {\em {Ann. Sci. \'{E}cole Norm. Sup. (4)}}, 18:553--561, 1985.

\bibitem{MR85}
A.~Medina and P.~Revoy.
\newblock Les groupes oscillateurs et leurs r\'{e}seaux.
\newblock {\em Manuscripta Math.}, 52(1-3):81--95, 1985.

\bibitem{MeRe93}
A.~{Medina} and P.~{Revoy}.
\newblock Alg\`{e}bres de lie orthogonales modules orthogonaux.
\newblock {\em {Comm. Algebra}}, 21(7):2295--2315, 1993.

\bibitem{Mostow}
G.~D. Mostow.
\newblock Fully reducible subgroups of algebraic groups.
\newblock {\em Amer. J. Math.}, 78:200--221, 1956.

\bibitem{Nikolayevsky}
Y.~Nikolayevsky.
\newblock Einstein solvmanifolds and the pre-{E}instein derivation.
\newblock {\em Trans. Amer. Math. Soc.}, 363(8):3935--3958, 2011.

\bibitem{Ov11}
G.~P. {Ovando}.
\newblock {Naturally reductive pseudo-Riemannian spaces}.
\newblock {\em {J. Geom. Phys.}}, 61(1):157--171, 2011.

\bibitem{Ovandosurvey}
G.~P. Ovando.
\newblock {Lie algebras with ad-invariant metrics. A survey - guide.}
\newblock {\em {Rend. Semin. Mat., Univ. Politec. Torino}}, 74(1):243--268,
  2016.

\bibitem{Samelson}
H.~Samelson.
\newblock {\em Notes on {L}ie algebras}.
\newblock Universitext. Springer-Verlag, New York, second edition, 1990.

\bibitem{Will:RankOne}
C.~Will.
\newblock Rank-one {E}instein solvmanifolds of dimension 7.
\newblock {\em Differential Geom. Appl.}, 19(3):307--318, 2003.

\bibitem{ZhZh01}
F.~{Zhu} and L.~{Zhu}.
\newblock {The uniqueness of the decomposition of quadratic Lie algebras}.
\newblock {\em {Commun. Algebra}}, 29(11):5145--5154, 2001.

\end{thebibliography}

\small\noindent D.~Conti and F.~A.~Rossi: Dipartimento di Matematica e Applicazioni, Universit\`a di Milano Bicocca, via Cozzi 55, 20125 Milano, Italy.\\
\texttt{diego.conti@unimib.it}\\
\texttt{federico.rossi@unimib.it}\medskip

\small\noindent V.~del Barco: IMECC - Universidade Estadual de Campinas, Rua Sergio Buarque de Holanda, 651, Cidade Universitaria Zeferino Vaz, 13083-859, Campinas, São Paulo, Brazil and CONICET, Argentina.\\
\texttt{delbarc@ime.unicamp.br}

\end{document}